\pdfoutput=1
\documentclass[11pt]{article}
\usepackage{amssymb}
\usepackage[all,cmtip]{xy}
\usepackage[width=17cm,top = 3.5cm,bottom=3.5cm]{geometry}
\usepackage{tikz}
\usepackage[all]{xy}
\usetikzlibrary{matrix,arrows,decorations.pathmorphing}
\usepackage{amsthm,amsfonts,amssymb,amsmath,amscd} 
\usepackage{aliascnt} 
\usepackage{mathrsfs} 
\usepackage{xargs,ifthen} 

\usepackage{comment}
\usepackage{color}

\usepackage{hyperref}
\definecolor{tocolor}{rgb}{.1,.1,.1}
\definecolor{urlcolor}{rgb}{.2,.2,.6}
\definecolor{linkcolor}{rgb}{.1,.1,.5}
\definecolor{citecolor}{rgb}{.4,.2,.1}
\hypersetup{backref=true, colorlinks=true, urlcolor=urlcolor, linkcolor=linkcolor, citecolor=citecolor}

\newcommandx{\thdef}[2]{
	\newaliascnt{#1}{theorem}  
	\newtheorem{#1}[#1]{#2}
	\aliascntresetthe{#1}  
	\newtheorem*{#1*}{#2}
	\expandafter\newcommand\expandafter{\csname #1autorefname\endcsname}{#2}
}

\makeatletter
\newtheorem*{rep@theorem}{\rep@title}
\newcommand{\newreptheorem}[2]{%
\newenvironment{rep#1}[1]{%
 \def\rep@title{#2 \ref{##1}}%
 \begin{rep@theorem}}%
 {\end{rep@theorem}}}
\makeatother

\newtheorem{theorem}{Theorem}[section]
\newreptheorem{theorem}{Theorem}

\thdef{lemma}{Lemma}
\thdef{corollary}{Corollary}
\thdef{conjecture}{Conjecture}
\newreptheorem{conjecture}{Conjecture}
\thdef{proposition}{Proposition}

\theoremstyle{definition}
\thdef{definition}{Definition}
\thdef{notation}{Notation}

\theoremstyle{remark}
\thdef{remark}{Remark}

\theoremstyle{remark}
\thdef{ex}{Example}

\newenvironment{example}
{\begin{ex}}%
{\hfill $\blacksquare$\end{ex}}

\newcommand{\spc}[1]{\mathsf{#1}} 
\newcommand{\shf}[1]{\mathcal{#1}} 

\newcommand{\CC}{\mathbb{C}}

\newcommand{\rbrac}[1]{\left(#1\right)} 

\newcommandx{\fn}[2][2=]{#1\ifthenelse{\equal{#2}{}}{}{\!\rbrac{{#2}}}} 
\newcommandx{\id}[2][2=]{\fn{{\rm id}_{#1}}[#2]} 

\newcommand{\ext}[2][\bullet]{\spc{\Lambda}^{#1}{#2}} 
\newcommandx{\End}[2][1=]{\fn{\spc{End}_{#1}}[#2]} 
\newcommandx{\Hom}[2][1=]{\fn{\spc{Hom}_{#1}}[#2]} 
\newcommandx{\Aut}[2][1=]{\fn{\spc{Aut}_{#1}}[#2]} 
\newcommandx{\image}[1]{\fn{\spc{img}}[#1]} 
\renewcommandx{\ker}[1]{\fn{\spc{ker}}[#1]} 
\newcommandx{\rank}[1]{\fn{\mathrm{rank}}[#1]} 

\newcommandx{\ann}[1]{\fn{\spc{ann}}[#1]} 

\newcommandx{\hlgy}[3][1=\bullet,3=]{\spc{H}_{#1}^{#3}\!\rbrac{{#2}}} 
\newcommandx{\cohlgy}[3][1=\bullet,3=]{\spc{H}^{#1}_{#3}\!\rbrac{{#2}}} 
\newcommandx{\chow}[3][1=\bullet,3=]{\spc{A}^{#1}_{#3}\!\rbrac{{#2}}} 

\newcommandx{\Ext}[3][1=\bullet,3=]{\fn{\spc{Ext}^{#1}_{#3}}[{#2}]} 
\newcommandx{\Tor}[3][1=\bullet,3=]{\fn{\spc{Tor}^{#1}_{#3}}[{#2}]} 

\newcommandx{\Pic}[1]{\fn{\spc{Pic}}[{#1}]} 
\newcommandx{\chernalg}[2][1=\bullet]{\fn{\spc{Chern}^{#1}}[{#2}]} 
\newcommandx{\chern}[2][1=]{\fn{c_{#1}}[#2]} 
\newcommandx{\ch}[2][1=]{\fn{\mathrm{ch}_{#1}}[{#2}]} 

\newcommandx{\sKer}[2][1=]{ \fn{ \shf{K}er_{#1}}[{#2}] } 
\newcommandx{\sHom}[2][1=]{ \fn{ \shf{H}om_{#1}}[{#2}] } 
\newcommandx{\sEnd}[2][1=]{ \fn{ \shf{E}nd_{#1}}[{#2}] } 
\newcommandx{\sExt}[3][1=\bullet,3=]{\fn{\shf{E}xt^{#1}_{#3}}[{#2}]} 
\newcommandx{\sTor}[3][1=\bullet,3=]{\fn{\shf{T}or^{#1}_{#3}}[{#2}]} 

\newcommandx{\forms}[2][1=\bullet]{\Omega^{#1}_{#2}} 
\newcommandx{\can}[1][1=]{\omega_{#1}} 
\newcommandx{\acan}[1][1=]{\omega_{#1}^{-1}} 
\newcommandx{\tshf}[1]{\shf{T}_{#1}} 
\newcommandx{\mvect}[2][1=\bullet]{ \ext[#1]{\tshf{#2}} }
\newcommandx{\der}[2][1=\bullet]{\mathscr{X}^{#1}_{#2}} 
\newcommandx{\sJet}[3][1=,2=]{\shf{J}^{#1}_{#2}#3} 

\newcommandx{\tb}[2][1=]{\spc{T}_{\!#1}{#2}} 
\newcommandx{\ctb}[2][1=]{\spc{T}_{\!#1}^*{#2}} 
\newcommandx{\lie}[2][2=]{\fn{\mathscr{L}_{#1}}[#2]} 
\newcommandx{\hook}[2][2=]{\fn{i_{#1}}[#2]} 

\newcommand{\del}{\partial}

\makeatletter
\newcommand{\thickbar}{\mathpalette\@thickbar}
\newcommand{\@thickbar}[2]{{#1\mkern1.5mu\vbox{
  \sbox\z@{$#1\mkern-1mu#2\mkern-1mu$}%
  \sbox\tw@{$#1\overline{#2}$}%
  \dimen@=\dimexpr\ht\tw@-\ht\z@-.6\p@\relax
  \hrule\@height.4\p@ 
  \vskip1\p@
  \hrule\@height.4\p@ 
  \vskip\dimen@
  \box\z@}\mkern1.5mu}
}
\makeatother

\newcommand{\JJ}{\mathbb J}

\newcommand{\delbar}{\overline\partial}
\newcommand{\eps}{\epsilon}

\newcommand{\ol}{\overline}

\renewcommand{\Re}{\mathrm{Re}}
\renewcommand{\Im}{\mathrm{Im}}

\def\ker{\text{ker}}

\def\End{\text{End}}
\def\log{\text{log}}

\newcommand{\bb}[1]{\mathbb{#1}}
\renewcommand{\cal}[1]{\mathcal{#1}}

\numberwithin{equation}{section}
\newtheoremstyle{parag}
  {\topsep}   
  {\topsep}   
  {}  
  {}       
  {\bfseries} 
  {.}         
  { } 
  {}          
\theoremstyle{parag}

\makeatletter
\def\@cite#1#2{{\normalfont[{#1\if@tempswa , #2\fi}]}}
\makeatother

\newcommand{\BPic}{\mathrm{Pic}^\cal{L}}
\renewcommand{\Pic}{\mathrm{Pic}}
\newcommand{\PG}{\mathcal{PG}}
\newcommand{\BPG}{\mathcal{PG}^\cal{L}}

\begin{document}

\title{\vspace{-4em} \huge Morita equivalence and the generalized K\"ahler potential}
\date{}

\author{
Francis Bischoff\thanks{Department of Mathematics, University of Toronto; {\tt bischoff@math.toronto.edu }}\and
Marco Gualtieri \thanks{Department of Mathematics, University of Toronto; {\tt mgualt@math.toronto.edu}}
 \and
Maxim Zabzine \thanks{Department of Physics and Astronomy, Uppsala University; {\tt maxim.zabzine@physics.uu.se}}
}
\maketitle

\abstract{
We solve the problem of determining the fundamental degrees of freedom underlying a generalized K\"ahler structure of symplectic type.  For a usual K\"ahler structure, it is well-known that the geometry is determined by a complex structure, a K\"ahler class, and the choice of a positive $(1,1)$-form in this class, which depends locally on only a single real-valued function: the K\"ahler potential.    Such a description for generalized K\"ahler geometry has been sought since it was discovered in 1984.  We show that a generalized K\"ahler structure of symplectic type is determined by a pair of holomorphic Poisson manifolds, a holomorphic symplectic Morita equivalence between them, and the choice of a positive Lagrangian brane bisection, which depends locally on only a single real-valued function, which we call the generalized K\"ahler potential.  Our solution draws upon, and specializes to, the many results in the physics literature which solve the problem under the assumption (which we do not make) that the Poisson structures involved have constant rank.  To solve the problem we make use of, and generalize, two main tools: the first is the notion of symplectic Morita equivalence, developed by Weinstein and Xu to study Poisson manifolds; the second is Donaldson's interpretation of a K\"ahler metric as a real Lagrangian submanifold in a deformation of the holomorphic cotangent bundle.  
}
\renewcommand{\contentsname}{}
\setcounter{tocdepth}{1}
\tableofcontents

\pagebreak
\section*{Introduction}

A central feature of K\"ahler geometry is the fact that, once a complex structure is chosen, the Riemannian metric depends locally on only a single real-valued function, the K\"ahler potential.  Indeed, in holomorphic coordinates we have the familiar expression for the metric tensor
\[
g_{i\bar\jmath} = \frac{\del^2 f}{\del z_i \del\bar z_j}.
\]
This fundamental observation is crucial to many aspects of K\"ahler geometry, especially to the problem of finding K\"ahler-Einstein metrics and to studying the K\"ahler-Ricci flow; in both cases these problems are reduced to differential equations for a single real-valued function.

The importance of K\"ahler structures in physics derives from the discovery by
Zumino~\cite{ZUMINO1979203} that if a Riemannian manifold is equipped with a K\"ahler structure, then the 2-dimensional sigma model, a field theory whose fields are maps from a Lorentzian 2-manifold to a Riemannian target manifold, may naturally be extended, by introducing additional fields and using the K\"ahler structure on the target, to a field theory with $\mathcal{N}=2$ supersymmetry, an enlargement of the usual Lorentz symmetry of the original model.  Zumino also made the key observation that the additional fields of the resulting theory could be interpreted, together with the original field, as the components of a single map called a superfield, but where the domain of the map is modified to be a supermanifold
 and the map satisfies a certain constraint.  
Zumino showed that in this more geometric formulation of the extended model, the Lagrangian density simplifies dramatically, and is given precisely by the pullback of the K\"ahler potential function $f$ described above.

It was later discovered by Gates, Hull and Ro\v{c}ek~\cite{MR776369} that K\"ahler geometry is not the only structure giving rise to such a supersymmetric extension; they showed that what is required is \emph{generalized K\"ahler geometry}, consisting of a pair $I_+, I_-$ of complex structures compatible with the Riemannian metric  and whose associated Hermitian forms $\omega_+, \omega_-$ are not necessarily closed.  Instead, they satisfy the conditions
\[
d_{+}^{c}\omega_{+} + d_{-}^{c}\omega_{-} = 0, \qquad dd_{\pm}^{c}\omega_{\pm} = 0.
\]
Gates, Hull and Ro\v{c}ek argued that while a naive application of the $\del\delbar$-lemma is impossible in this case, there should still be a reformulation of the sigma model generalizing the one found by Zumino, for which the Lagrangian density would again be a single real-valued function, which, in turn, determines the metric.  The authors established this under the assumption that the complex structures $I_+, I_-$ commute, and in subsequent works~\cite{lindstrom2007potential,lindstrom2007generalized} it was established under the assumption that the commutator $[I_+,I_-]$ has constant rank.  
The purpose of this paper is to introduce a new mathematical approach to this problem, and to solve the general case, where the rank of $[I_+,I_-]$ is not locally constant, a common situation in examples. It is important to emphasize that in this paper, we do make the simplifying assumption that $I_++I_-$ is invertible (what we call ``symplectic type''); in the language of superfields, we assume there are chiral and semichiral superfields, but no twisted chiral superfields. 


We now give a brief outline of the main ideas and results in the paper.  
In Section~\ref{symptyp}, we review the key tools from Dirac geometry and generalized complex geometry which we use, and we define the notion of a generalized K\"ahler structure of symplectic type to which all our main results apply.  Most importantly, we explain Hitchin's observation that the commutator $[I_+, I_-]$ defines a pair of holomorphic Poisson structures $\sigma_+, \sigma_-$ relative to the complex structures $I_+, I_-$.  

In Section~\ref{Donkahler}, we revisit the classical K\"ahler case and explain an important observation of Donaldson that the K\"ahler potential function has a global interpretation as a Lagrangian submanifold of a holomorphic symplectic affine bundle deforming the cotangent bundle and determined by the K\"ahler class.  Our main insight is that it is this viewpoint on K\"ahler geometry which may be profitably generalized.  

The space which plays the role of the holomorphic symplectic affine bundle in our more general construction is a \emph{Morita equivalence} bibundle relating the Poisson structures $\sigma_+$ and $\sigma_-$.  Morita equivalence is a relation between Poisson manifolds introduced by Weinstein and Xu~\cite{weinstein1987symplectic,xu1991morita}, and involves a symplectic manifold which maps via Poisson maps to the spaces being related; it may be viewed as an invertible generalized morphism. We review the key properties of this equivalence relation on Poisson manifolds in Section~\ref{moreq}.  

In Section~\ref{branbi}, we prove our main result, Theorem~\ref{main}, showing that generalized K\"ahler structures of symplectic type are equivalent to holomorphic symplectic Morita equivalences equipped with a bisection which is Lagrangian for the imaginary part of the holomorphic symplectic form (i.e., a Lagrangian brane).  We also explain the curious way in which such a Lagrangian brane determines a symmetric tensor---the generalized K\"ahler metric. 

Having expressed the geometry in terms of a Lagrangian brane in a holomorphic symplectic manifold, we are able in Section~\ref{secpot} to generalize the classical generating function technique to obtain a local description in terms of a single real-valued function, the generalized K\"ahler potential.   We verify that our potential reduces to those found earlier in special cases, and we use it to produce a new example (Proposition~\ref{globex}) of a generalized K\"ahler metric which has a global generalized K\"ahler potential, yet whose Poisson structures are not of constant rank.  

In Section~\ref{picsec}, we explore the basic features of the category of holomorphic symplectic Morita equivalences equipped with brane bisection. Real versions of this category or \emph{Picard groupoid} have been investigated recently~\cite{bursztyn2004picard, bursztyn2015picard}, and so we carefully explain the unexpected link to generalized K\"ahler geometry. 

Since our main result represents a generalized K\"ahler metric as a Lagrangian submanifold, it is clear that the group of Hamiltonian flows must act on the space of generalized K\"ahler metrics. In Section~\ref{hamflo}, we develop this idea into a method for deforming generalized K\"ahler structures. We also show that this flow method specializes to the existing constructions~\cite{MR1702248,MR2217300,MR2371181,gualtieri2010branes} of generalized K\"ahler metrics.  Finally, in Section~\ref{locdef}, we use the above results to prove, in Theorem~\ref{GKlocallyflow}, that any generalized K\"ahler structure of symplectic type may locally be constructed via the flow method applied to a natural ``degenerate'' generalized K\"ahler structure which exists on any holomorphic Poisson manifold.  In other words, we have the striking result that, locally, any generalized K\"ahler structure of symplectic type may be obtained by applying the flow construction to a holomorphic Poisson structure. 

\vspace{.05in}

\noindent \textbf{Acknowledgements.} We would like to thank Nigel Hitchin, Chris Hull, Ulf Lindstr\"om, Martin Ro\v{c}ek, David Mart\'ines Torres, and Rikard von Unge for discussions on this subject over many years.  F.B. is supported by an NSERC CGS Doctoral award, M.G. is supported by an NSERC Discovery Grant, and M.Z. is supported by Vetenskapsr\r{a}det under grant W2014-5517, by the STINT grant and by the grant ``Geometry and Physics'' from the Knut and Alice Wallenberg foundation.

%
%
%
%
%
\pagebreak
\section{Generalized K\"ahler structures of symplectic type}\label{symptyp}
We start by reviewing the concept of Generalized K\"ahler (GK) geometry and its many reformulations, as well as the subclass of GK structures of \emph{symplectic type} which will be the main subject of this paper. We begin with the notion of bihermitian geometry, first formulated by Gates, Hull, and Ro\v{c}ek in the context of $N = (2,2)$ supersymmetry \cite{MR776369}. A \emph{bihermitian} structure on a manifold $M$ consists of the data $(g, I_{+}, I_{-})$ of a Riemannian metric $g$ and two complex structures $I_{\pm}$ such that $g$ is Hermitian with respect to both complex structures and such that the following integrability conditions are satisfied:
\[
d_{+}^{c}\omega_{+} + d_{-}^{c}\omega_{-} = 0, \qquad dd_{\pm}^{c}\omega_{\pm} = 0,
\]
where $\omega_{\pm} = gI_{\pm}$ are the associated Hermitian forms and $d_{\pm}^{c} = i(\ol{\partial}_{\pm} - \partial_{\pm})$ are the real operators determined by the two complex structures. The closed $3$-form $H = \pm d_{\pm}^{c}\omega_{\pm}$ is used to define the Wess--Zumino--Witten term in the non-linear sigma model. We assume that it is trivial in cohomology and choose an explicit potential $b$ such that $H = -db$, called the $B$-field. Hence we are concerned with the following data: $(g, I_{+}, I_{-}, b)$. Note that these data admit an action of the additive group of closed $2$-forms $\Omega^{2, cl}(M)$:
\[
B : (g, I_{+}, I_{-}, b) \mapsto (g, I_{+}, I_{-}, b + B),
\]
which we refer to as $B$-field gauge transformations.

As explained in \cite{gualtieri-2010}, this geometry admits a reformulation in terms of commuting \emph{generalized complex structures} on $TM \oplus T^{\ast}M$ known as \emph{generalized K\"ahler geometry}. Let us briefly recall the relevant notions. The bundle $\bb{T} = TM \oplus T^{\ast}M$ admits a natural split-signature non-degenerate symmetric pairing (let $X, Y \in \Gamma(TM)$ and $\xi, \eta \in \Gamma(T^{\ast}M)$):
\[
\langle X + \xi, Y + \eta \rangle = \frac{1}{2}(\xi(Y) + \eta(X)),
\]
as well as an extension of the Lie bracket, known as the \emph{Courant bracket}:
\[
[\![ X + \xi, Y + \eta ]\!] = [X,Y] + \cal{L}_{X}(\eta) - \iota_{Y}(d\xi),
\]
which satisfies the Jacobi identity but is skew-symmetric only up to an exact term involving the above pairing. A generalized complex (GC) structure is defined to be an endomorphism $\bb{J}$ of $\bb{T}$ squaring to $-1$ which is compatible with the symmetric pairing and such that its $+i$ eigenbundle $L$ in the complexification $\bb{T}_\bb{C} = \bb{T}\otimes\bb{C}$ is closed under the (complexified) Courant bracket. A generalized K\"ahler structure is then given by a pair $(\bb{J}_{\cal{A}}, \bb{J}_{\cal{B}})$ of commuting generalized complex structures such that the product $\bb{G} = - \bb{J}_{\cal{A}} \bb{J}_{\cal{B}}$ defines a positive definite metric on $\bb{T}M$, which is to say that for all non-zero $v \in \bb{T}M$:
\[
\langle \bb{G} v, v \rangle > 0.
\]
The bihermitian structure $(g,I_{+},I_{-},b)$ may be encoded in the following Generalized K\"ahler pair:
\begin{align*}
\bb{J}_{\cal{B}} &= \frac{1}{2} e^{b} \begin{pmatrix} I_{+} + I_{-} & -(\omega_{+}^{-1} - \omega_{-}^{-1}) \\  \omega_{+} - \omega_{-} & -(I_{+}^{\ast} + I_{-}^{\ast}) \\ \end{pmatrix} e^{-b}, \\
\bb{J}_{\cal{A}} &= \frac{1}{2} e^{b} \begin{pmatrix} I_{+} - I_{-} & -(\omega_{+}^{-1} + \omega_{-}^{-1}) \\  \omega_{+} + \omega_{-} & -(I_{+}^{\ast} - I_{-}^{\ast}) \\ \end{pmatrix} e^{-b},
\end{align*}
where the conjugation is by the natural orthogonal symmetry determined by $b$:
\[
e^{b} = \begin{pmatrix} 1 & 0 \\ b & 1 \end{pmatrix}.
\]
Conversely, given a GK pair $(\bb{J}_{\cal{A}}, \bb{J}_{\cal{B}})$, it is possible to recover the bihermitian data in the following way: the product $\bb{G}= - \bb{J}_{\cal{A}}\bb{J}_{\cal{B}}$ has eigenvalues $\pm 1$, with corresponding eigenbundles \[
C_{\pm} = \{X + (b\pm g)X \ |\ X\in TM\},
\]
which determine $g$ and $b$. These eigenbundles map isomorphically to $TM$ by the natural projection, and $I_{\pm}$ is given by the restriction of $\bb{J}_{\cal{B}}$ to $C_{\pm}$. Note that $B$-field gauge transformations act on the $GK$ structure by conjugation:
\[
B: (\bb{J}_{\cal{A}}, \bb{J}_{\cal{B}}) \mapsto (e^{B}\bb{J}_{\cal{A}}e^{-B}, e^{B}\bb{J}_{\cal{B}}e^{-B}).
\]

The GC structures are themselves fully specified by their $+i$ eigenbundles $L_{\cal{A}}, L_{\cal{B}} \subset \bb{T}_\CC$, which are complex \emph{Dirac structures}: maximal isotropic subbundles of $\bb{T}_\bb{C}$ which are involutive for the Courant bracket. As such, it is possible to completely characterize the conditions for generalized K\"ahler geometry in terms of the Dirac structures $L_{\cal{A}}$ and $L_{\cal{B}}$, leading to yet another reformulation of the geometry. This is what is done in \cite{mgualt-hamdef}, and we reproduce the key statements here.  For this, we need two basic observations concerning Dirac structures.  

The first is that holomorphic Poisson structures may be viewed as Dirac structures in the following way.  Given a complex structure $I$ on the manifold $M$ and a holomorphic Poisson tensor $\sigma$, we decompose it into real and imaginary parts 
$
\sigma = -\frac{1}{4}(IQ + iQ),
$
and define the following GC structure:
\[
\bb{J}_{\sigma} =  \begin{pmatrix} 
      -I & Q \\
      0 & I^{\ast} \\
   \end{pmatrix}.
\]
The $+i$ eigenbundle of $\bb{J}_\sigma$ is the complex Dirac structure
\[
L_{\sigma} = \{ X + \sigma(\zeta) + \zeta \  | \  X \in T^{0,1}M, \  \zeta \in T_{1,0}^{*}M \}.
\]
Note that this Dirac structure encodes both the complex structure $I$ and the Poisson tensor $\sigma$; $L_{\sigma}$ is involutive if and only if $I$ is integrable, $\sigma$ is holomorphic, and $\sigma$ is Poisson.  Also, because the intersection with the complexified tangent bundle $T_\CC = TM\otimes\CC$ is given by $L_{\sigma} \cap T_\CC = T^{0,1}$, the Dirac structure $L_\sigma$ satisfies the special property that $T_\bb{C} = (L_{\sigma} \cap T_\bb{C}) \oplus (\ol{L}_{\sigma} \cap T_\bb{C})$. It is shown in \cite{mgualt-hamdef} that this property is in fact enough to guarantee that a complex Dirac structure arises from a holomorphic Poisson structure in the manner described above.

Second, we recall the notion of linear combination of Dirac structures. Given a Dirac structure $L \subset \bb{T}_\bb{C}$, we may scale it by $\lambda \in \bb{C}^*$ to produce a new Dirac structure:
\[
\lambda L = \{ X + \lambda \alpha \ | \ X + \alpha \in L \}.
\]
Similarly, given a pair of Dirac structures $L_{1}$ and $L_{2}$ such that their projections to $T_\bb{C}$ are transverse, we define the sum Dirac structure:
\[
L_{1} + L_{2} = \{ X + \alpha + \beta \ | \ X + \alpha \in L_{1}, X + \beta \in L_{2} \}.
\]
Since it occurs frequently, we use the notation $L_1-L_2$ to denote $L_1 + (-1)L_2$.  Any closed 2-form $B$ defines a Dirac structure through its graph $\Gamma_B = \{X + i_XB\ |\ X\in TM\}$; a key property of the Dirac sum is that addition of $\Gamma_B$ coincides with gauge transformation:  for any Dirac structure $L$, 
\begin{equation}\label{gaugesum}
e^B(L) = L + \Gamma_B.
\end{equation}
\begin{proposition}\label{diffdir}\cite[Proposition 4.3]{mgualt-hamdef}
A pair $L_{\cal{A}}, L_{\cal{B}} \subset \bb{T}M \otimes \bb{C}$ of complex Dirac structures defines a GK structure precisely when the following three properties are satisfied:
\begin{enumerate}
\item Both Dirac structures define GC structures, which is to say that $L_{\cal{A}} \cap \ol{L}_{\cal{A}} = 0$ and  $L_{\cal{B}} \cap \ol{L}_{\cal{B}} = 0$.
\item The complex Dirac structures 
\[
L_{\sigma_{+}} = \frac{i}{2} (\ol{L}_{\cal{B}} - \ol{L}_{\cal{A}}), \qquad L_{\sigma_{-}} = \frac{i}{2} (\ol{L}_{\cal{B}} - L_{\cal{A}}),
\]
define holomorphic Poisson structures $(I_{+}, \sigma_{+}), (I_{-}, \sigma_{-})$ respectively.
\item For all nonzero $u \in L_{\cal{A}} \cap L_{\cal{B}}$, we have $\langle u, \ol u \rangle > 0$.
\end{enumerate}
\end{proposition}
The complex structures $I_\pm$ arising in the Proposition are indeed those of the original bihermitian data, and the Poisson tensors $\sigma_{\pm}$ are those first observed by Hitchin in \cite{MR2217300}. They share a common imaginary part which can be expressed in terms of the original bihermitian data $(g, I_{+}, I_{-})$. More precisely, the holomorphic Poisson structures are given by $\sigma_{\pm} = -\frac{1}{4}(I_{\pm}Q + iQ)$, where $Q = \frac{1}{2}[I_{-},I_{+}] g^{-1}$. 


\subsection{Symplectic type}\label{symptype}
In this paper, we consider only GK structures of symplectic type, defined as follows:
\begin{definition}
A generalized K\"ahler structure $(\JJ_\cal{A},\JJ_\cal{B})$ is of \emph{symplectic type} if $\JJ_\cal{A}$ is gauge equivalent to a symplectic structure. That is,  there is a closed 2-form $\beta$ such that 
\[
e^{\beta}  \JJ_\cal{A} e^{-\beta}  = \begin{pmatrix} 
      0 & -F^{-1} \\
      F & 0 \\
   \end{pmatrix},
\]
where $F$ is a symplectic form.
\end{definition}
The $+i$-eigenbundle of such a $\JJ_\cal{A}$ is then
\begin{equation}\label{ella}
L_{\cal{A}} = \Gamma_{(\beta - iF)},
\end{equation} 
which, using Proposition~\ref{diffdir}, leads to holomorphic Poisson structures given by 
\begin{equation}\label{sumsigma}
L_{\sigma_{+}} =\frac{i}{2}(\ol{L}_{\cal{B}} - \Gamma_{\beta+iF}),   
\qquad 
L_{\sigma_{-}} =\frac{i}{2}(\ol{L}_{\cal{B}} - \Gamma_{\beta-iF}).  
\end{equation}
Since the Dirac sum specializes to a gauge transformation~\eqref{gaugesum}, the above immediately implies that the two Poisson structures are in fact gauge equivalent by the action of the symplectic form:
\[
L_{\sigma_{+}} = e^{F} L_{\sigma_{-}}.
\]
If we unpack what this means in terms of the data $(I_{+}, I_{-}, Q, F)$ of Proposition~\ref{diffdir}, we arrive at the following two equations, first considered in \cite{gualtieri2010branes}:
\begin{align} \label{star1}
&I_{+} - I_{-} = QF, \\
&FI_{+} + I_{-}^{\ast}F = 0. \label{star2}
\end{align} 
Due to the fact that Dirac sum by $\Gamma_{\beta\pm iF}$ is invertible, we see from~\eqref{sumsigma} that we may reconstruct the entire GK structure from $(L_{\sigma_{+}}, L_{\sigma_{-}}, F, \beta)$. For instance, the metric $g$ and $B$-field $b$ of the bihermitian geometry are given as follows:
\begin{equation} 
g = -\frac{1}{2}F(I_{+} + I_{-}),\qquad \ b = \beta - \frac{1}{2}F(I_{+} - I_{-}), \label{metricF}
\end{equation}
and the Dirac structure $L_{\cal{B}}$ is given by 
\begin{equation}
L_{\cal{B}} = e^{\beta + iF}(2i \ol{L}_{\sigma_{-}}). \label{LB}
\end{equation}
On the other hand, it is clear that equations \ref{star1} and \ref{star2} are not sufficient to guarantee that the data $(L_{\sigma_{+}}, L_{\sigma_{-}}, F, \beta)$ defines a GK structure: for example, if $I_{+} + I_{-}$ fails to be invertible, then we can see from equation \ref{metricF} that we will not get a metric. For this reason, we make the following definition:
\begin{definition}
A \emph{degenerate} GK structure of symplectic type consists of the data $(L_{\sigma_{+}}, L_{\sigma_{-}}, F, \beta)$ of two holomorphic Poisson structures $(I_{\pm}, \sigma_{\pm})$ and two closed $2$-forms $F, \beta \in \Omega^{2, cl}(M,\bb{R})$ such that 
\[
L_{\sigma_{+}} = e^{F} L_{\sigma_{-}}.
\]
\end{definition}
\begin{example}\label{trivdeg}
Associated to any holomorphic Poisson structure $(I, \sigma)$ there is a canonical degenerate GK structure given by $(L_{\sigma}, L_{\sigma}, 0, 0)$.
\end{example}
Given a degenerate GK structure, we may construct complex Dirac structures $L_{\cal{A}}, L_{\cal{B}}$ according to equations \ref{ella} and \ref{LB} respectively. Then $L_{\cal{A}}$ defines a GC structure if and only if $F$ is non-degenerate (i.e. symplectic) and $L_{\cal{B}}$ defines a GC structure if and only if $I_{+} + I_{-}$ is invertible. In the case that both these conditions hold then we get a GK structure but with a metric that may be indefinite. Therefore, in order to get a genuine GK structure it is necessary to require that the symmetric tensor $g$ defined through equation \ref{metricF} is positive-definite, in which case we say that $F$ is \emph{positive}. This summarizes the content of the following theorem:
\begin{theorem} \cite[Theorem 6.2] {gualtieri2010branes} \label{GK=star}
There is a bijection between Generalized K\"ahler structures of symplectic type $(L_{\cal{A}}, L_{\cal{B}})$ and degenerate GK structures of symplectic type $(L_{\sigma_{+}}, L_{\sigma_{-}}, F, \beta)$ such that the $2$-form $F$ is positive.
\end{theorem}

We will find it useful to study the larger class of degenerate GK structures, applying the condition of positivity of the form $F$ when we want to return to the setting of GK geometry. Also, the action of $B$-field gauge transformations is particularly simple in the case of degenerate GK structures: it acts only (and transitively) on the closed two-form $\beta$:
\[
B : (L_{\sigma_{+}}, L_{\sigma_{-}}, F, \beta) \mapsto (L_{\sigma_{+}}, L_{\sigma_{-}}, F, \beta + B).
\]
In the remainder of this paper we will fix the gauge such that $\beta = 0$. Hence when speaking of degenerate GK structures, we will only specify the data $(L_{\sigma_{+}}, L_{\sigma_{-}}, F)$, or equivalently, $(I_{+}, I_{-}, Q, F)$.

The following is a useful result giving us another means of extracting the metric $g$ and the $B$-field $b$ from the data $(I_{+}, I_{-}, Q, F)$:
\begin{lemma} \label{metric11}
The Hermitian forms $\omega_{\pm} = g I_{\pm}$ associated to the bi-Hermitian data $(g, I_{\pm})$ coincide with the $(1,1)$ components of the symplectic form $F$ relative to the complex structures $I_\pm$:
\[
\omega_{\pm} = F^{(1,1)_{\pm}}.
\]
As a result, we obtain expressions for the metric and $B$-field
\[
g = - F^{(1,1)_{\pm}} I_{\pm},\qquad b = \mp F^{(2,0) + (0,2)_{\pm}} I_{\pm}.
\] 
\begin{proof}
We prove the result for the $I_{-}$ complex structure; the other case is completely analogous. Starting with \eqref{metricF} and using~\eqref{star2}, we have
\[
\omega_{-} = gI_{-} = \frac{1}{2}F(1 - I_{+} I_{-}) = \frac{1}{2}(F + I_{-}^{\ast}FI_{-}) = F^{(1,1)_{-}}.
\] 
The remaining statement is then another application of~\eqref{metricF}.
%
\end{proof}
\end{lemma}
\begin{remark}
The above Lemma holds true even in the case that $g$ and $F$ are degenerate. 
\end{remark}

\section{Introduction to the problem: the K\"ahler case}\label{Donkahler}

It is fundamental to K\"ahler geometry that, given the underlying complex manifold, one can describe the K\"ahler metric locally by specifying a single real-valued function: the K\"ahler potential.  For this reason, we may think of the complex moduli and the potential function as independent degrees of freedom comprising the K\"ahler structure.
The main problem we solve is to determine the analogous degrees of freedom inherent in Generalized K\"ahler geometry. More precisely, we give an answer to the following two questions in the setting of GK structures of symplectic type:
\begin{enumerate}
\item What is the holomorphic structure underlying a Generalized K\"ahler manifold?
\item What additional data is needed to specify the Riemannian metric, and how is it determined locally by a real-valued function?
\end{enumerate}
To understand the methods we use to solve these questions, it is helpful to revisit them in the usual K\"ahler case, where our approach coincides with a reformulation of K\"ahler geometry first considered by Donaldson in his study of the complex Monge-Amp\`{e}re equation \cite{donaldson2002holomorphic}.

Let $X = (M,I)$ be a complex manifold equipped with a K\"ahler metric $g$. Notice that this defines a GK structure of symplectic type where $I_{+} = I_{-} = I$ and $F = \omega=gI$ is the K\"ahler form. The Hitchin Poisson structure in this case vanishes: $Q = \frac{1}{2}[I,I] g^{-1} = 0.$  We begin by explaining how the K\"ahler structure defines a deformation of the holomorphic cotangent bundle $T^\ast X$.
First, cover the manifold $X$ by open sets $U_{i}$ on which a K\"ahler potential $K_{i} \in C^{\infty}(U_{i}, \bb{R})$ is chosen: $\omega = i \partial \ol{\partial} K_{i}$. On the double overlap $U_{i} \cap U_{j}$ we have $\partial \ol{\partial}(K_{i} - K_{j}) = 0$ and so $\mu_{ij} = i\del(K_i-K_j)$ defines a closed holomorphic $(1,0)$-form.  Therefore, we can define a holomorphic affine transformation on the double overlap:
 \[
 A_{ij} : T^{\ast}U_{i}|_{U_{i} \cap U_{j}} \longrightarrow T^{\ast}U_{j}|_{U_{i} \cap U_{j}},\qquad \alpha_{x} \mapsto \alpha_{x} + \mu_{ij}(x).
 \]
In fact, $A_{ij}$ is also a symplectomorphism: if $\Omega_0$ is the canonical symplectic form on $T^\ast X$, then
\[
A_{ij}^{\ast} (\Omega_0) = \Omega_0 + \pi^{\ast} \mu_{ij}^{\ast}(\Omega_0) = \Omega_0 + \pi^{\ast} d\mu_{ij} = \Omega_0,
\]
where $\pi : T^{\ast}X \to X$ is the vector bundle projection. Since $A_{ij}$ defines a cocycle, we may define an affine bundle $Z$ modelled on $T^{\ast}X$ as follows:
\[
Z = (\coprod_{i} T^{\ast}U_{i}) /\sim, \ \ (\alpha_{x})_{i} \sim (A_{ij}(\alpha_{x}))_{j}.
\]
Then $Z$ inherits a symplectic form $\Omega$, making it into a holomorphic symplectic manifold. In fact, more is true: since the symplectic form on $T^{\ast}X$ is compatible with the additive structure, it follows that the form $\Omega$ is compatible with the action of $T^{\ast}X$ on $Z$, in the sense that the graph of the action map $A : T^{\ast}X \times_{X} Z \to Z$ in $(T^{\ast}X, \Omega_0) \times (Z, \Omega) \times (Z, - \Omega)$ is a holomorphic Lagrangian submanifold. This holomorphic symplectic affine bundle $(Z, \Omega)$ is what encodes the underlying holomorphic structure of the K\"ahler manifold as well as the K\"ahler class.

To see where the metric comes in, observe that the potentials $K_{i}$ provide us with a global section of $Z \to X$. More precisely, if we define $\cal{L}_{i} = - i \partial K_{i} : U_{i} \to T^{\ast}U_{i}$, then because 
\[
A_{ij} \circ \cal{L}_{i}(x) = -i \partial K_{i} + \mu_{ij} =  -i \partial K_{i} + i \partial(K_{i} - K_{j}) = \cal{L}_{j},
\]
we get a global section $\cal{L} : X \to Z$. This section fails to be holomorphic, but it has the interesting property of being Lagrangian with respect to $\Im(\Omega)$ and symplectic with respect to $\Re(\Omega)$. Indeed, computing the pullback we see that 
\[
\cal{L}^{\ast}(\Omega) = \cal{L}_{i}^{\ast}(\Omega_0) = d(-i\partial K_{i}) = i \partial \ol{\partial}K_{i} = \omega,
\]
recovering the K\"ahler form and hence the metric $g$. Submanifolds such as the image of $\cal{L}$, which are Lagrangian with respect to $\Im(\Omega)$ and which provide sections of the projection $Z \to X$, are what we will come to call \emph{brane bisections}. They are what encode the extra data needed to recover the metric in GK geometry.

Let us note one more thing about the affine bundle $Z$: using the brane $\cal{L}$ we can define a diffeomorphism between the underlying smooth manifolds
\[
\psi : T^{\ast} X \to Z, \ \  \alpha_{x} \mapsto \alpha_{x} + \cal{L}(x).
\]
Using the above mentioned fact that the symplectic forms are compatible with the affine bundle structure we see that 
\[
\psi^{\ast} \Omega = \Omega_0 + \pi^{\ast} \cal{L}^{\ast} \Omega = \Omega_0 + \pi^{\ast} \omega.
\]
In other words, the affine bundle $(Z, \Omega)$ is nothing but a \emph{twisted cotangent bundle} $(T^{\ast}X, \Omega_0 + \omega)$, and the brane bisection $\cal{L}$ corresponds simply to the zero section. 

Having separated the K\"ahler degrees of freedom into the holomorphic moduli of $(Z,\Omega)$ and the choice of a smooth brane $\cal{L}$, it is natural to study the deformations of $\cal{L}$ keeping $(Z,\Omega)$ fixed. Any deformation $\cal{L}' : X \to Z$ of $\cal{L}$ is given by the graph of a $(1,0)$-form $\alpha$. Pulling back the symplectic form, we get
\[
\omega' := (\cal{L}')^{\ast} \Omega = \alpha^{\ast} (\Omega_0 + \pi^{\ast} \omega) = \omega + d\alpha.
\]
That $\cal{L}'$ is a brane imposes the condition that $\omega'$ is real; using the $\del\delbar$-lemma, we conclude that $\omega' - \omega = i \partial \ol{\partial} f$, for a real-valued function $f$. In short, varying $\mathcal{L}$ with $(Z,\Omega)$ fixed is equivalent to varying $\omega$ within its K\"ahler class. 

\begin{proposition}\cite[Section~2.]{donaldson2002holomorphic}
On a K\"ahler manifold, the cohomology class $[\omega]$ of the K\"ahler form determines a holomorphic symplectic affine bundle $(Z, \Omega)$ modelled on the cotangent bundle, and the metric $g$ determines a smooth section $\cal{L}$ of the bundle $Z$ which is symplectic for $\Re(\Omega)$ and Lagrangian for $\Im(\Omega)$.  Conversely, this data $(Z,\Omega,\cal{L})$ uniquely determines the K\"ahler structure.  Under this correspondence, deforming $\cal{L}$ is equivalent to varying $\omega$ within the K\"ahler class.
\end{proposition}

As we will see, the situation is much the same in the general case of GK geometry of symplectic type. We will again be able to encode the GK geometry in terms of a brane bisection in a holomorphic symplectic manifold. The crucial difference will be that the symplectic manifold, far from being a twisted cotangent bundle, is determined by the Hitchin Poisson structure according to the theory of symplectic Morita equivalence.
%
%
%
%

\section{Holomorphic symplectic Morita equivalence}\label{moreq}
%
In order to access the holomorphic structure underlying GK geometry, we need to begin with the underlying holomorphic Poisson structures $\sigma_{\pm}$. In this section, we recall the necessary tools from Poisson geometry. 

A fundamental feature of Poisson geometry is that a Poisson structure $Q$ determines a (singular) foliation of the underlying manifold $M$ by \emph{symplectic leaves}. These leaves are equivalence classes of points related by Hamiltonian flows. Taking the quotient of $M$ by this equivalence relation defines a generalized space, which can be suitably upgraded to a differentiable stack. Viewed in this way, i.e., as stacks, Poisson manifolds inherit a more general notion of morphism than the usual Poisson maps. 
Such generalized morphisms between the associated stacks may be conveniently described in terms of spans of Poisson maps, as follows. A generalized morphism between Poisson manifolds $(M_{1}, Q_{1})$ and $(M_{2}, Q_{2})$ is given by a symplectic manifold $(S, \omega)$, together with Poisson maps 
\[
\xymatrix{ &(S, \omega)\ar[dl]\ar[dr] & \\(M_{1}, Q_{1}) & &(M_{2}, -Q_{2}) }
\]
satisfying a list of properties (see \cite{MR2166451} for more details). Note that a symplectic manifold $(S, \omega)$ equipped with a Poisson map to $(M, Q)$ is called a \emph{symplectic realization}, a concept introduced by Weinstein in \cite{weinstein1983local}. Poisson manifolds, spans of symplectic realizations, and isomorphisms between spans assemble into a $2$-category. The equivalences in this category are the main objects of interest in our study of GK geometry of symplectic type. These are known as \emph{Morita equivalences} and were first introduced by Xu \cite{xu1991morita} who gave the following definition:

\begin{definition}[Morita equivalence] A Morita equivalence between Poisson manifolds $(M_{1}, Q_{1})$ and $(M_{2}, Q_{2})$ is a symplectic manifold $(S, \omega)$ together with a pair of surjective submersions $\pi_{1} : S \to M_{1}, \pi_{2} : S \to M_{2}$ with connected and simply-connected fibres such that 
\begin{enumerate}
\item $\pi_{1}$ is Poisson and $\pi_{2}$ is anti-Poisson;
\item the vertical distributions $\ker(d\pi_{1})$ and $\ker(d\pi_{2})$ are symplectically orthogonal; 
\item $\pi_{1},\pi_2$ are \emph{complete} in the sense that the pullback of any complete Hamiltonian vector field is complete.
\end{enumerate}
We view this as a morphism from $(M_{2}, Q_{2})$ to $(M_{1}, Q_{1})$.
\end{definition}
The identity Morita self-equivalence $(\Sigma(M), \Omega)$ of a Poisson manifold $(M, Q)$ comes equipped with an isomorphism $\Sigma(M) \circ \Sigma(M) \cong \Sigma(M)$, reflecting the fact that composition with the identity is trivial up to isomorphism. This endows the space $\Sigma(M)$ with the structure of a Lie groupoid, making it into a \emph{symplectic groupoid}, a notion due to Karas\"ev, Weinstein and Zakrzewski \cite{MR854594,weinstein1987symplectic,MR1081010,MR1081011}.
\begin{definition}[Symplectic groupoid] A symplectic groupoid $(\cal{G} \rightrightarrows M, \Omega)$ is a Lie groupoid $\cal{G}$ equipped with a symplectic structure $\Omega$ such that the graph of the multiplication map is a Lagrangian submanifold of $\cal{G} \times \bar{\cal{G}} \times \bar{\cal{G}} $, where $\bar{\cal{G}}$ denotes $\cal{G}$ with the opposite symplectic form $-\Omega$.
\end{definition}
\begin{remark}
A symplectic form on $\cal{G}$ satisfying the above property is called a \emph{multiplicative form}. Equivalently, $\Omega$ satisfies the following equation on the fibre product $\cal{G} \times_{M} \cal{G}$
\[
m^{\ast} \Omega = p_{1}^{\ast} \Omega + p_{2}^{\ast} \Omega,
\]
where $m : \cal{G} \times_{M} \cal{G} \to \cal{G}$ is the groupoid multiplication, and $p_{i} : \cal{G} \times_{M} \cal{G} \to \cal{G}$ are the two projections.
\end{remark}
Given a Poisson manifold, the existence of an identity Morita self-equivalence is governed by whether a certain Lie algebroid determined by the Poisson structure is integrable to a Lie groupoid.  
For this reason we call $(\Sigma(M), \Omega)$ an \emph{integration} of the Poisson manifold, and we call the Poisson manifold \emph{integrable} if such an integration exists. In the category of smooth Lie groupoids this problem was solved by Crainic and Fernandes \cite{crainic2003integrability}, inspired by the work of Cattaneo and Felder \cite{cattaneo2001poisson}. In the holomorphic category, this problem was studied by Laurent-Gengoux, Sti\'{e}non and Xu \cite{laurent2009integration}, who showed that a holomorphic Poisson manifold is integrable if and only if its underlying real or imaginary part is an integrable real Poisson structure. We assume that all our Poisson manifolds are integrable. A given Poisson manifold may have several different integrations, but there is a unique \emph{source simply connected} one, characterized by the fact that the fibers of the source map are connected and simply connected \cite[Proposition
6.8]{MR2012261}; this is the identity Morita self-equivalence $\Sigma(M)$ above, and we refer to it as the \emph{Weinstein groupoid}. 

Let $(S, \omega)$ be a Morita equivalence between Poisson manifolds $(M_{1}, Q_{1})$ and $(M_{2}, Q_{2})$. Since composition with the identity Morita equivalence is trivial up to isomorphism, we obtain isomorphisms $\Sigma(M_{1}) \circ S \cong S$ and $S \circ \Sigma(M_{2}) \cong S$. These define actions of the groupoids $\Sigma(M_{i})$ on the space $S$, endowing $S$ with the structure of a \emph{bi-principal groupoid bi-bundle} for the pair $(\Sigma(M_{1}), \Sigma(M_{2}))$. Most importantly for our purposes, these actions are symplectic in the sense that the graphs of the action maps are Lagrangian submanifolds in the product $\Sigma(M_{i}) \times S \times \bar{S}$, where $\bar{S}$ denotes $S$ with the opposite symplectic form. In this way, we make contact with the theory of Morita equivalence for symplectic Lie groupoids. The following theorem of Xu tells us that the notions of Morita equivalence for symplectic groupoids and Poisson manifolds agree.
 \begin{theorem}\cite[Theorem 3.2]{xu1991morita}
Two integrable Poisson manifolds $(M_{1}, Q_{1})$ and $(M_{2}, Q_{2})$ are Morita equivalent if and only if their Weinstein groupoids $\Sigma(M_{1}), \Sigma(M_{2})$ are (symplectically) Morita equivalent. 
\end{theorem}

\subsection{Examples} We end this section with examples of symplectic groupoids and Morita equivalences.

\begin{example}[Zero Poisson structure]
The Weinstein groupoid integrating the zero poisson structure on a manifold $M$ is given by the cotangent bundle $T^{\ast}M$ with its canonical symplectic form and groupoid multiplication given by fibrewise addition. 
\end{example}

\begin{example}[Symplectic manifold] \label{symplectic manifold groupoid}
The Weinstein groupoid of a simply connected symplectic manifold $(M, \omega)$ is given by the \emph{pair groupoid} $Pair(M) = (M \times M, \omega \oplus -\omega)$. The source and target maps are given by the two projections and the multiplication is given by the formula $(a,b)*(b,c) = (a,c)$. If $M$ is not simply connected, the pair groupoid still provides an integration of $(M,\omega)$, but now the Weinstein groupoid is given by the fundamental groupoid $\Pi(M)$, with symplectic form $\Omega = t^{\ast}\omega - s^{\ast}\omega$, where $t$ and $s$ are the target and source maps, respectively.
\end{example}

\begin{example}[Lie-Poisson structure]
Let $\mathfrak{g}$ be a Lie algebra. Its dual $\mathfrak{g}^{\ast}$ has a natural Poisson bracket extending the Lie bracket on $\mathfrak{g}$. 
If $G$ is a Lie group integrating $\mathfrak{g}$ then $T^{*}G$ with its canonical symplectic form gives a symplectic groupoid integrating $\mathfrak{g}^{\ast}$. The source and target maps are given by left and right trivialization respectively, and the multiplication is given by the following formula:
\[
\xi_{g_{1}} \star \eta_{g_{2}} := d(R_{g_{2}^{-1}})^{\ast}_{g_{1}g_{2}}(\xi_{g_{1}}) = d(L_{g_{1}^{-1}})^{\ast}_{g_{1}g_{2}}(\eta_{g_{2}}),
\]
where $R_{g}$ and $L_{g}$ are right and left multiplication, respectively. By using left or right trivialization we can see that this groupoid is isomorphic to the action groupoid induced by the coadjoint action of $G$ on $\mathfrak{g}^{\ast}$. See \cite{coste1987groupoides} for further details.
\end{example}
\begin{example}\label{affpois}
As a concrete instance of the previous example, consider the linear Poisson structure on $\bb{C}^{2}$ given by $\Pi = x \partial_{x} \wedge \partial_{y}$. This is a Poisson structure corresponding to the Lie algebra $\frak{g} = \bb{C}x \oplus \bb{C}y$ with bracket $[x,y] = x$. The Lie group integrating this algebra is the group $G$ of affine transformations of $\bb{C}$ (or rather the universal cover of this). Therefore we can obtain the symplectic groupoid integrating our Poisson structure as the action groupoid for the coadjoint action of $G$ on $\frak{g}^{*}$.

However, in the present case a more direct approach to integration is possible: note that the Hamiltonian vector fields corresponding to the coordinate functions are given by $X_{x} = x \partial_{y}$ and $X_{-y} = x \partial_{x}$. The vector field $x \partial_{x}$ generates the multiplicative action of $\bb{C}$: $(a,x) : x \mapsto e^{a}x$, and the vector field $x \partial_{y}$ generates the additive action of $\bb{C}$ rescaled by a factor of $x$: $(b,y): y \mapsto y + xb$. Combining the two actions we get the action groupoid $\cal{G} = \bb{C}^{2} \ltimes \bb{C}^{2}$, with coordinates $(a,b,x,y)$, with the source and target maps given respectively by 
 \[
 s(a,b,x,y) = (x,y) \qquad t(a,b,x,y) = (e^{a}x, y + xb),
 \]
and with the multiplication given by 
\[
(a_{1}, b_{1}, x_{1}, y_{1}) \star (a_{2}, b_{2}, x_{2}, y_{2}) = (a_{1} + a_{2}, b_{1} e^{a_{2}} + b_{2}, x_{2}, y_{2}).
\]
Note that indeed $\cal{G} = \widetilde{\text{Aff}(\bb{C})} \ltimes \bb{C}^{2}$: the action groupoid for the action of (the universal cover of) the affine group on the dual of its Lie algebra. 

To compute the symplectic form on $\cal{G}$ we make use of the following fact: in the case of an invertible Poisson structure with corresponding symplectic form $\omega$, the symplectic form on the integrating groupoid is given by $\Omega = t^{\ast} \omega - s^{\ast} \omega$ (see Example~\ref{symplectic manifold groupoid} above). In the present case the Poisson structure is invertible on a dense subset, corresponding to the meromorphic symplectic form
\[
\omega = \frac{1}{x} dx \wedge dy.
\] 
Hence, it follows that the symplectic form on the groupoid $\cal{G}$ is given by the same formula:
\begin{align*}
\Omega &= t^{\ast} \omega - s^{\ast} \omega \\
 &=  \frac{1}{e^{a}x}d(e^{a}x)\wedge d(y + xb) - \frac{1}{x} dx \wedge dy \\
&= da \wedge d(y + xb) - db \wedge dx.
\end{align*}
See \cite{radko2006picard} for further discussion of this example in the smooth category. 
\end{example}

We now consider examples of Morita equivalences. 
\begin{example}
Recall from the discussion above that the Weinstein groupoid of a Poisson manifold $(M,Q)$ provides the trivial Morita self-equivalence of $Q$. Therefore all of the above examples of groupoids also give examples of Morita equivalences.
\end{example}
\begin{example}
Let $S_{1}$ and $S_{2}$ be simply connected symplectic manifolds. Then $S_{1} \leftarrow S_{1} \times \bar{S_{2}} \to S_{2}$ defines a Morita equivalence. More generally, two symplectic manifolds are Morita equivalent if and only if they have the same fundamental group. See \cite{xu1991morita} for details.
\end{example}

\begin{example}
Morita equivalences can be composed with Poisson diffeomorphisms. Let $\sigma : M_{2} \to M_{1}$ be a Poisson diffeomorphism and let $M_{2} \xleftarrow{\pi_{2}} S \xrightarrow{\pi_{3}} M_{3}$ be a Morita equivalence. Then we get the following Morita equivalence
\[
\xymatrix{ &S\ar[dl]_-{\sigma \circ \pi_{2}}\ar[dr]^-{\pi_{3}} & \\M_{1} & &M_{3} }
\]
In particular, by composing Poisson diffeomorphisms with the trivial Morita equivalence we get a map from Poisson diffeomorphisms to Morita equivalences.
\end{example}

\begin{example}\label{btrans}
Let $(M,Q)$ be a Poisson manifold. If $B \in \Omega^{2, cl}(M)$ is a closed $2$-form such that $id + B \circ Q: T^{\ast}M \to T^{\ast}M$ is invertible then we can define a new Poisson structure 
\[
Q^{B} := Q \circ (id + B \circ Q)^{-1},
\]
which is called the \emph{$B$-field transform} of $Q$. Bursztyn and Radko have shown in \cite{bursztyn2003gauge} that in this case $(M,Q)$ and $(M, Q^{B})$ are Morita equivalent in the following way 
\[
\xymatrix{ &(\Sigma(M), \Omega + t^{\ast}B)\ar[dl]_-t\ar[dr]^-s & \\(M, Q^{B}) & &(M, Q) }
\]
where $(\Sigma(M), \Omega)$ is the Weinstein groupoid of $(M,Q)$ and $s$ and $t$ are the source and target maps, respectively. In particular, if we apply this construction to the zero Poisson structure $(M, Q= 0)$ then we get a Morita self-equivalence given by $(T^{\ast}M, \Omega_0 + p^{*}B)$, where $\Omega_0$ is the canonical symplectic form and $p : T^{\ast}M \to M$ is the projection.
\end{example}

\section{Generalized K\"ahler metrics as brane bisections}\label{branbi}
For a GK structure of symplectic type, we have seen that there are two underlying holomorphic Poisson structures $\sigma_\pm$.  These Poisson manifolds are not biholomorphically equivalent in general; in fact, the underlying complex manifolds $(M, I_\pm)$ may not even be isomorphic.  But, in analogy to Example~\ref{btrans}, it was shown in~\cite{bailey2016integration} that they are actually holomorphically Morita equivalent.

\begin{proposition}\cite[Proposition 6.4]{bailey2016integration} \label{Star->ME}
Let $(I_{+}, I_{-}, Q, F)$ be a degenerate GK structure of symplectic type, and let $\sigma_{\pm} = -\frac{1}{4}(I_{\pm}Q + iQ)$ denote the corresponding holomorphic Poisson structures on the complex manifolds $X_{\pm} := (M, I_{\pm})$. Then $(X_{+}, \sigma_{+})$ and $(X_{-}, \sigma_{-})$ are holomorphically symplectically Morita equivalent in a canonical way.
\end{proposition}

We recall the construction of the canonical Morita equivalence as it will be one of our main objects of study. It can be broken up into two steps as follows:
\begin{enumerate}
\item Integrate $(X_{-}, \sigma_{-})$ to its holomorphic symplectic groupoid $(\Sigma(X_{-}), \Omega_{-})$. 
\item Deform $\Sigma(X_-)$ to the holomorphic symplectic manifold 
\begin{equation}\label{omegaf}
(Z,\Omega) = (\Sigma(X_-), \Omega_{-} + t^{*}F),
\end{equation}
 where $t$ is the target map of $\Sigma(X_-)$.
\end{enumerate}
Note that by modifying the symplectic form, we are also modifying the underlying complex structure. To see this, recall that a $(2,0)$-form can be written as $\Omega = \omega I + i \omega$, where $I$ is the complex structure and $\omega$ is  real. Therefore, specifying the form $\Omega$ and insisting that it be holomorphic symplectic is enough to specify a new complex structure.  With respect to this deformed holomorphic symplectic structure, the source map $s$ remains holomorphic and anti-Poisson, while the target map $t$ is now holomorphic and Poisson as a map from $(Z,\Omega)$ to $(X_{+}, \sigma_{+})$. We use the notation $\pi_{+} = t$ and $\pi_{-} = s$ for these maps from the modified domain $(Z,\Omega)$. Proposition~\ref{Star->ME} then states that the diagram
\[
\xymatrix{ &(Z,\Omega)\ar[dl]_{\pi_+}\ar[dr]^{\pi_-} & \\(X_{+},\sigma_{+}) & &(X_{-}, \sigma_{-}) }
\]
defines a Morita equivalence between $(X_{-},\sigma_{-})$ and $(X_{+}, \sigma_{+})$. 

This result shows that underlying any GK structure of symplectic type is a holomorphic symplectic Morita equivalence between holomorphic Poisson structures.  In order to recover the GK structure from this Morita equivalence, we need some real (i.e. non-holomorphic) data.  To see how this arises, observe that since $Z$ was obtained by deforming the holomorphic symplectic groupoid $\Sigma(X_-)$, it contains a distinguished submanifold $\cal{L}$ coming from the identity bisection.  This submanifold is neither holomorphic nor Lagrangian in $(Z,\Omega)$. It is, however, characterized by the following two properties: 

\begin{enumerate}
\item The submanifold $\cal{L}\subset Z$ is a smooth section of both $\pi_{-}$ and $\pi_{+}$; in other words, it is a smooth \emph{bisection}. This defines a diffeomorphism between the underlying smooth manifolds of $X_+$ and $X_-$, which is required because a GK structure involves two complex structures living on the same real manifold. 
\item The bisection $\cal{L}$ is Lagrangian with respect to $\omega = \Im(\Omega)$: from the point of view of generalized complex geometry, this is known as a \emph{brane} in $(Z, \Omega)$. Notice that~\eqref{omegaf} implies that $\Omega|_{\cal{L}}=F$, so that we recover the real closed 2-form $F$.
\end{enumerate}

We record the above properties satisfied by $\cal{L}$ in the following definition.

\begin{definition}
A \emph{brane bisection} in the Morita equivalence $(Z,\Omega)$ is a smooth submanifold of $Z$ which is Lagrangian for $\Im(\Omega)$ and which is a section of both $\pi_{-}$ and $\pi_{+}$. 
\end{definition}

\begin{theorem} \label{main}
A degenerate GK structure of symplectic type $(I_{+}, I_{-}, Q, F)$ is equivalent to a holomorphic symplectic Morita equivalence with brane bisection $(Z, \Omega, \cal{L})$ between the holomorphic Poisson structures $\sigma_{\pm} =  -\frac{1}{4}(I_{\pm}Q + iQ)$.
\begin{proof}
One direction of the theorem follows from Proposition~\ref{Star->ME} and the above remarks. For the converse direction let us start with a holomorphic symplectic Morita equivalence with brane bisection $(Z, \Omega = B + i\omega, \cal{L})$ and construct a degenerate GK structure. Such a Morita equivalence goes between two holomorphic Poisson manifolds $(X_{\pm}, \sigma_{\pm} = -\frac{1}{4}(I_{\pm} Q_{\pm} + iQ_{\pm}))$. Now observe that $(Z, \omega)$ defines a real smooth symplectic Morita equivalence between the real Poisson structures $(X_{+}, Q_{+})$ and $(X_{-}, Q_{-})$. The fact that $\cal{L}$ is a brane bisection means precisely that it defines a Lagrangian bisection in this Morita equivalence, and therefore that it defines a Poisson diffeomorphism between the two Poisson structures. The upshot of this is that when we use the bisection to identify $X_{-} = \cal{L} = X_{+}$, then $Q_{-} = Q_{+}$. We denote this real Poisson structure by $Q$ and the underlying smooth manifold by $M$. 

Now that we have made these identifications $(Z, \omega)$ defines a self-Morita equivalence of $(M, Q)$ and $\cal{L}$ is a Lagrangian bisection inducing the identity diffeomorphism on $M$. The symplectic groupoid $(\Sigma(X_{-}), \Omega_{-} = B_{-} + i \omega_{-})$ of $(X_{-}, \sigma_{-})$ acts principally on $(Z, \Omega)$, and so the brane $\cal{L}$ induces the following isomorphism of smooth symplectic Morita equivalences
\[
\phi : (\Sigma(X_{-}), \omega_{-}) \to (Z, \omega), \ g \mapsto \lambda(t(g)) \ast g,
\]
where $\lambda : M \to Z$ is the section of $\pi_{-}$ induced by $\cal{L}$. Using $\phi$ to compare the real parts of $\Omega_{-}$ and $\Omega$ we see that 
\begin{equation}
\phi^{\ast}( B) = B_{-} + t^{\ast} (B|_{\cal{L}}). \label{B-eq}
\end{equation}
Let $F := B|_{\cal{L}}$. Then we see that under this isomorphism 
\[
\phi^{\ast} (\Omega) = \Omega_{-} + t^{\ast}(F).
\]
On the manifold $M$, we now have the data of the two complex structures $I_{\pm}$ coming from the spaces $X_{\pm}$, the real Poisson structure $Q$, and a closed 2-form $F$. We then show that $(I_{+}, I_{-}, Q, F)$ satisfy equations \ref{star1} and \ref{star2}. For simplicity, we use $\phi$ to identity $(Z, \omega)$ with $(\Sigma(X_{-}), \omega_{-})$. On this space, we have complex structures $I$ and $I_-$ coming from $Z$ and $\Sigma(X_{-})$, respectively. Note that the complex structure $I$ satisfies
\[
t_{*} I = I_{+} t_{*},\qquad s_{*} I = I_{-} s_{*}.
\]
In order to verify equations \ref{star1} and \ref{star2}, we show that they hold for $(I, I_{-}, \omega^{-1}, t^{*}(F))$ on $\Sigma(X_{-})$ and therefore on $M$ by push-forward.  First we note that since $B = \omega I$ and $B_{-} = \omega I_{-}$, we have from equation \ref{B-eq} that $\omega (I - I_{-}) = t^{\ast} (F)$, which we can rearrange to equation \ref{star1}:
\[
I - I_{-} = \omega^{-1} t^{\ast} (F). 
\]
Applying the target projection we get 
\[
(I_{+} - I_{-}) t_{*} = t_{*} \omega^{-1} t^{*} F t_{*} = Q F t_{*},
\]
which implies equation \ref{star1} on $M$ since $t_{*}$ is surjective. Note that the $2$-form $t^{*}(F)$ gives rise to the map $t^{\ast}Ft_{\ast}:T \to T^{\ast}$. Equation \ref{star2} on the groupoid follows from a direct computation: 
\[
I^{*} t^{*} (F) + t^{*} (F) I_{-} = I^{*} \omega (I - I_{-}) + \omega (I - I_{-}) I_{-} =0,
\]
where we have used the identities $t^{*}(F) = \omega(I - I_{-})$ and $I^{*} \omega = \omega I$. But this implies the corresponding equation on $M$ since 
\[
t^{*} (I_{+}^{*} F + F I_{-} )= I^{*} t^{*}( F) + t^{*}( F) I_{-} = 0,
\]
and $t^{*}$ is injective on differential forms. This establishes equations \ref{star1} and \ref{star2} on $M$. It is clear that the two constructions outlined are inverse to each other, so we get the desired equivalence.
\end{proof}
\end{theorem}

Theorem~\ref{main} answers the main questions we raised in Section~\ref{Donkahler}, identifying the Morita equivalence $(Z,\Omega)$ as the underlying holomorphic data of a GK manifold of symplectic type, and showing how the additional data of a brane bisection $\cal{L}$ specifies the GK metric.  In Section~\ref{secpot}, we explain how $\cal{L}$ may be locally given by a real-valued potential function.


\subsection{Induced metric}
We now discuss how the metric and $B$-field of the GK structure arise from the holomorphic symplectic Morita equivalence with brane bisection $(Z, \Omega, \cal{L})$. From the equivalence established in Theorem~\ref{main}, we know that a Morita equivalence with brane bisection gives rise to a degenerate GK structure, and from this we can extract the data of a (possibly degenerate) metric and $B$-field. In this section, we detail how to determine the metric and $B$-field directly from the geometry of the Morita equivalence.

Given a real $2$-form $F$ and a complex structure $I$, the tensor $FI$ decomposes into symmetric and anti-symmetric parts: $FI = S + A$, where 
\[
S = F^{(1,1)}I,\qquad A = F^{(2,0) + (0,2)}I.
\]
Given a Morita equivalence with brane bisection $(Z, \Omega, \cal{L})$, we have two induced complex structures $I_{\pm}$ on the brane, as well as the real 2-form $F = \Omega|_\cal{L}$. Therefore the tensors $FI_{\pm}$ give rise to symmetric tensors $S_{\pm} = F^{(1,1)_{\pm}} I_{\pm}$ and anti-symmetric tensors $A_{\pm} = F^{(2,0) + (0,2)_{\pm}}I_{\pm}$.  
From the theory of GK structures we know to expect that $S_{+} = S_{-}$ and $A_{+} = - A_{-}$. Indeed, from Lemma~\ref{metric11} we know that  $g = - S_{\pm}$ and $b = \mp A_{\pm}$. Here we explain these facts directly.

\begin{proposition}
Let $(Z, \Omega, \cal{L})$ be a holomorphic symplectic Morita equivalence with brane bisection, let $I_{\pm}$ denote the complex structures induced on the brane, and let $F = \Omega|_{\cal{L}}$ denote the real $2$-form on the brane obtained by pulling back the symplectic form. If $S_{\pm} = F^{(1,1)_{\pm}} I_{\pm}$ and $A_{\pm} = F^{(2,0) + (0,2)_{\pm}}I_{\pm}$, then 
\[
S_{+} = S_{-} \ \ \text{and } \ A_{+} = - A_{-}.
\]
\begin{proof}
Let $\Omega = B + i \omega $ be the holomorphic symplectic form on $Z$, so that $F = B|_{\cal{L}}$. Writing $A_{\pm}$ and $S_{\pm}$ as the (anti-)symmetrizations of $FI_{\pm}$, we have
\[
2S_{\pm}(u,v) = B(I_{\pm}u, v) + B(I_{\pm}v,u),\qquad 2A_{\pm}(u,v) = B(I_{\pm}u, v) - B(I_{\pm}v,u),
\] 
for $u,v \in T\cal{L}$. Using the two direct sum decompositions of $TZ$ along $\cal{L}$ determined by the fibres of the projections $\pi_{\pm}$, 
\[
TZ|_{\cal{L}} = T\cal{L} \oplus K_{\pm},\qquad K_{\pm} =  \ker(d \pi_{\pm}),
\]
we can write the following expressions for $I_{\pm}$:
\[
I u = I_{\pm}u + k_{\pm}u,\qquad I_{\pm}u \in T\cal{L}, \ k_{\pm}u \in K_{\pm},
\]
for $u \in T\cal{L}$. Then using the fact that $K_{\pm}$ are $\omega$-symplectic orthogonal, $T\cal{L}$ is $\omega$-Langrangian, and $\omega I = I^{*} \omega$, we get the following identity
\begin{align*}
0 &=\omega(k_+u,k_-v)= \omega(Iu - I_{+}u, Iv - I_{-} v) 
= \omega(Iv, I_{+}u) - \omega(Iu, I_{-}v),
\end{align*}
for $u, v \in T\cal{L}$. Using the relation $B = \omega I$, this gives the following identity:
\[
B(I_{+}u, v) = B(I_{-}v, u),
\]
which implies the two identities above for $S_{\pm}$ and $A_{\pm}$.
\end{proof}
\end{proposition}

The upshot of the present discussion is as follows: when we pullback the holomorphic $(2,0)$-form $\Omega$ along a brane we get a real $2$-form $F$ which is no longer $(2,0)$ with respect to either of the induced complex structures $I_{\pm}$. As such, it gives rise to symmetric and anti-symmetric tensors on the brane. The fact that $F$ was obtained by pulling back the symplectic form of a holomorphic Morita equivalence implies that we end up with a single symmetric and anti-symmetric tensor; these give rise to the metric and $B$-field of the GK structure. The positivity or non-degeneracy of the metric arises from the configuration of the brane in the Morita equivalence. For example, $g$ is non-degenerate if and only if $I(T\cal{L}) \cap K = 0 = I(T\cal{L}) \cap T\cal{L} $, where $K$ is the average of $K_{\pm}$ with respect to $T\cal{L}$. Indeed, the condition $I(T\cal{L}) \cap T\cal{L} = 0$ is equivalent to invertibility of $F$ and $I(T\cal{L}) \cap K = 0$ is equivalent to invertibility of $I_{+} + I_{-}$. Recall that these two conditions imply that we get a GK structure with non-degenerate but possibly indefinite metric.

\subsection{Examples}
We end this section by providing examples of generalized K\"ahler structures of symplectic type and their corresponding holomorphic symplectic Morita equivalences with brane bisection.

\begin{example}[Holomorphic Poisson structure]\label{integtriv}
Let $(M, I, \sigma)$ be a Poisson manifold, and let $Q = -4 \Im (\sigma)$. As in Example~\ref{trivdeg}, the associated degenerate GK structure is given by the data $(I, I, Q, 0)$. The associated Morita equivalence with brane bisection is given by the holomorphic symplectic groupoid $(\Sigma(X), \Omega)$ integrating $\sigma$, viewed as the trivial Morita self-equivalence, with brane bisection given by the identity bisection $\epsilon$.
\end{example}

\begin{example}[K\"ahler Manifold]\label{kahman}
A K\"ahler manifold $(M, I, g, \omega)$ defines a generalized K\"ahler structure of symplectic type for which $I_{\pm} = I$, $Q = 0$, and $F = \omega$. Therefore the Morita equivalence is given by the holomorphic cotangent bundle of $X = (M, I)$ twisted by $\omega$:
\[
(Z, \Omega) = (T^{\ast}X, \Omega_0 + \pi^{\ast}\omega),
\]
where $\pi: T^{*}X \to X$ is the projection and $\Omega_0$ is the canonical symplectic form. The brane bisection $\cal{L}$ is given by the zero section viewed inside of $Z$. As discussed in Section~\ref{Donkahler}, $(Z, \Omega)$ is an affine bundle modelled on the holomorphic cotangent bundle of $X$, and the brane bisection is both symplectic with respect to $\Re(\Omega)$ and Lagrangian with respect to $\Im(\Omega)$. Conversely, any such affine bundle $\cal{A} \to X$ gives a Morita equivalence and any section $\cal{L}$ of $\pi = \pi_{+} = \pi_{-}$ which is symplectic for $\Re(\Omega)$ and Lagrangian for $\Im(\Omega)$ automatically defines a non-degenerate bilinear form on $M$.  This is due to the fact that $K = \ker(d \pi)$ is a complex Lagrangian, which implies that $I(T\cal{L}) \cap \ker (d\pi) = 0$.
\end{example}

\begin{example}[Hyper-K\"ahler Manifold] \label{hyperKahler}
Let $(M, I, J, K, g)$ be a hyper-K\"ahler structure, and let 
\[
(\omega_{I}, \omega_J, \omega_K) = (gI, gJ, gK)
\]
be the corresponding triple of K\"ahler forms. This data defines the GK structure of symplectic type 
\[
(I_{+}, I_{-}, Q, F) = (I, J, \omega_{K}^{-1}, \omega_{I} + \omega_{J}).
\]
The induced holomorphic Poisson structures are non-degenerate, with corresponding holomorphic symplectic forms
\[
\Omega_{+} = \omega_{J} + i \omega_{K},\qquad \Omega_{-} = -\omega_{I} + i \omega_{K}.
\]
Let $X_{\pm} = (M, I_{\pm})$. Then, assuming that $M$ is simply-connected (or using a non source-simply connected integration), we  have a holomorphic Morita equivalence given by 
\[
(Z, \Omega) = (X_{+}, \Omega_{+}) \times (X_{-}, -\Omega_{-}),
\]
with $\pi_{+}$ and $\pi_{-}$ given by the projections onto the first and second factors respectively. The brane bisection $\cal{L}$ is then given by the diagonally embedded copy of $M$:
\[
\cal{L} = \{ (m, m) \ | \ m \in M \} \subseteq Z.
\]
\end{example}

\section{The generalized K\"ahler potential}\label{secpot}
In this section, we discuss an application of our main result to the problem of locally describing a generalized K\"ahler structure in terms of a real-valued function analogous to the K\"ahler potential. There is evidence in the physics literature for the existence of such a \emph{generalized K\"ahler potential} (see for example \cite{zabzine2009generalized, lindstrom2007generalized, hull2012generalized, lindstrom2007potential}). However, all previous constructions only work in a neighbourhood where the Poisson structure $Q$ is regular (i.e. has constant rank). As we will see, by viewing a GK structure as a brane in a holomorphic symplectic Morita equivalence, one can describe the geometry in terms of a real-valued function even in a neighbourhood where the Poisson structure changes rank. The fundamental reason for this is that the brane $\cal{L}$ is a Lagrangian submanifold of the real symplectic manifold $(Z, \Im(\Omega))$, and so may be described using a generating function.

Our generalized K\"ahler potential is obtained in the following way.  Let $(Z, \Omega, \cal{L})$ be a holomorphic symplectic Morita equivalence with brane bisection.  Choose a \emph{Darboux chart} of $(Z, \Omega)$, i.e. a local holomorphic symplectomorphism 
\[
(T^{\ast}L, \Omega_0) \to (Z, \Omega),
\]
where $L\subset Z$ is a holomorphic Lagrangian submanifold and $\Omega_0$ is the canonical holomorphic symplectic form on the holomorphic cotangent bundle $T^*L$. 
Using the identification of $T^*L$ with the bundle of $(1,0)$-forms, the brane $\cal{L}$ can be viewed as a smooth submanifold of $T^{\ast}_{1,0}L$. We require that the Darboux chart be \emph{generic} in the sense that $\cal{L}$ is the graph of a $(1,0)$-form $\eta \in \Omega^{1,0}(L)$. Then, since 
\[
d \eta = \eta^{\ast} \Omega_0 = \Omega|_{\cal{L}} = F
\]
is a real closed 2-form, we have that $d( \Im (\eta)) = 0$, which locally implies that $\Im(\eta) = -\frac{1}{2}dK$, for $K \in C^{\infty}(L,\bb{R})$ a smooth real-valued function on $L$, uniquely determined up to a real constant. And since $\eta$ is a $(1,0)$-form, its real part is determined by its imaginary part, so that $\eta = - i \partial K$ and, finally,
\[
F = i \partial \bar \partial K.
\]
 Therefore we call $K$ the \emph{generalized K\"ahler potential} for the GK structure, and we record these observations in the following theorem. 

\begin{theorem} \label{GK potential}
Let $(Z, \Omega, \cal{L})$ be a holomorphic symplectic Morita equivalence with brane bisection. Given a generic Darboux chart $(T^{\ast}L, \Omega_0) \to (Z, \Omega)$, there is a smooth real-valued function $K \in C^{\infty}(L, \bb{R})$, unique up to an additive real constant, such that $\cal{L} = Gr(- i \partial K)$. Furthermore, the real $2$-form $F = \Omega|_{\cal{L}}$ is given by $F = i \partial \bar \partial K$.  We call $K$ the generalized K\"ahler potential.
\end{theorem}

\subsection{Examples}

In the first example, we see that the generalized K\"ahler potential does indeed specialize to the classical K\"ahler potential. 

\begin{example}[K\"ahler potential] \label{Kahlerex} Suppose the Poisson structure vanishes: $Q = 0$. Then a GK structure of symplectic type is automatically K\"ahler, and as we saw in Example~\ref{kahman}, the Morita equivalence is given by a twist of the holomorphic cotangent bundle by the K\"ahler form $\omega$:
\[
(Z, \Omega) = (T^{\ast}X, \Omega_0 + \pi^{\ast}\omega),
\]
and the brane bisection $\cal{L}$ is given by the zero section of $T^{\ast}X$ viewed as a brane in $Z$. Now choose a local holomorphic Lagrangian section $\lambda$ of $\pi: Z \to X$ with domain $U\subset X$ and image $L$. The projection $\pi$ defines a holomorphic isomorphism from $L$ to $U$, and so we can use the natural action of $T^{\ast}X$ on $Z$ to define the following holomorphic Darboux chart:
\[
(T^{\ast}U, \Omega_0) \to (Z, \Omega), \ \alpha_{x} \mapsto \alpha_{x} + \lambda(x),
\]
where $\alpha_{x} \in T^{\ast}_{x}U$. Theorem \ref{GK potential} then tells us that $\cal{L} = Gr(-i \partial K)$ and $\omega = i \partial \bar \partial K$, for a real-valued function $K \in C^{\infty}(U, \bb{R})$, recovering the usual notion of K\"ahler potential.
\end{example}

In the case where the underlying Poisson structure $Q$ is regular, it is possible to choose the Darboux chart in such a way as to recover the GK potentials occurring in the physics literature:

\begin{example}[Non-degenerate Poisson structure] If the Poisson structure $Q = \omega^{-1}$ is non-degenerate, then both of the holomorphic Poisson structures $\sigma_\pm$ are as well, with corresponding holomorphic symplectic forms 
\[
\Omega_{\pm} = \omega I_{\pm} + i \omega.
\]
The Morita equivalence between the resulting holomorphic symplectic manifolds is then  
\[
(Z, \Omega) = (X_{+}, \Omega_{+}) \times (X_{-}, -\Omega_{-}),
\]
with $\pi_{+}$ and $\pi_{-}$ the projections onto the first and second factors, respectively. The brane bisection $\cal{L}$ is given by the diagonally embedded manifold. We now construct a holomorphic Darboux chart on $Z$ from a pair of holomorphic Darboux charts for $(X_{\pm}, \Omega_{\pm})$: let $(q^{\alpha},p_{\alpha})_{\alpha = 1}^{n}$ define a holomorphic Darboux chart on $X_{+}$ so that $\Omega_{+} = dp_{\alpha} \wedge dq^{\alpha}$, and let $(Q^{\alpha},P_{\alpha})_{\alpha = 1}^{n}$ define a holomorphic Darboux chart on $X_{-}$ so that $\Omega_{-} = dP_{\alpha} \wedge dQ^{\alpha}$. Then $(q^{\alpha},p_{\alpha},Q^{\alpha},P_{\alpha})_{\alpha = 1}^{n}$ defines a holomorphic chart on $Z$ with respect to which the form $\Omega$ has the following expression:
\[
\Omega = dp_{\alpha} \wedge dq^{\alpha} + dQ^{\alpha} \wedge dP_{\alpha}.
\]
A natural choice for the complex Lagrangian in $(Z, \Omega)$ is given by $L = \{ Q^{\alpha} = p_{\alpha} = 0 \}$, so that $ (q^{\alpha},p_{\alpha},Q^{\alpha}, P_{\alpha})_{\alpha = 1}^{n} $ define a holomorphic Darboux chart with $q^{\alpha}$ and $P_{\alpha}$ coordinates along $L$ and $p_{\alpha}$ and $Q^{\alpha}$ the fibre coordinates. Theorem \ref{GK potential} then tells us that in the generic situation, the brane is given in these coordinates by the graph of $-i \partial K$ for a real-valued function $K(q^{\alpha},P_{\alpha}, \bar{q}^{\alpha}, \bar{P}_{\alpha})$:
\[
\cal{L} = \{ (q^{\alpha},p_{\alpha},Q^{\alpha},P_{\alpha})_{\alpha = 1}^{n} \ |\  p_\alpha=- i \frac{\partial K}{\partial q^{\alpha}},\ Q^\alpha=  - i \frac{\partial K}{\partial P_{\alpha}}
\},
\]
and the symplectic form is given by $F = i \partial \bar \partial K$.
Using $q^\alpha$ and $P_\alpha$ as coordinates on this brane, we can then express all of the remaining data in terms of the GK potential $K$ as follows. First, the two projection maps $\pi_{\pm}$ have the following coordinate expressions
\[
\pi_{+}|_{\cal{L}}(q^{\alpha},P_{\alpha}) = (q^{\alpha}, -i \frac{\partial K}{\partial q^{\alpha}}), \qquad \pi_{-}|_{\cal{L}}(q^{\alpha}, P_{\alpha}) = (-i \frac{\partial K}{\partial P_{\alpha}}, P_{\alpha}), 
\]
and therefore the holomorphic symplectic forms $\Omega_{\pm}$ can be expressed as follows
\begin{align*}
\Omega_{+} &= i \Big( \frac{\partial^{2} K}{\partial \bar{q}^{\beta} \partial q^{\alpha}}  dq^{\alpha} \wedge d\bar{q}^{\beta}+ \frac{\partial^{2} K}{\partial P_{\beta} \partial q^{\alpha}} dq^{\alpha} \wedge dP_{\beta} + \frac{\partial^{2} K}{\partial \bar{P}_{\beta} \partial q^{\alpha}} dq^{\alpha} \wedge d\bar{P}_{\beta} \Big) \\
\Omega_{-} &= i \Big( \frac{\partial^{2} K}{\partial \bar{P}_{\beta} \partial P_{\alpha}} d\bar{P}_{\beta} \wedge dP_{\alpha} + \frac{\partial^{2} K}{\partial q^{\beta} \partial P_{\alpha}} dq^{\beta} \wedge dP_{\alpha}+ \frac{\partial^{2} K}{\partial \bar{q}^{\beta} \partial P_{\alpha}} d\bar{q}^{\beta} \wedge dP_{\alpha} \Big).
\end{align*}
Since these forms are holomorphic symplectic, they determine the complex structures $I_{\pm}$.  The symplectic form $\omega$ is the common imaginary part of $\Omega_{\pm}$:
\[
\omega = \frac{1}{2}\Big(\frac{\partial^{2} K}{\partial P_{\beta} \partial q^{\alpha}} dq^{\alpha} \wedge dP_{\beta} + \frac{\partial^{2} K}{\partial \bar{P}_{\beta} \partial q^{\alpha}} dq^{\alpha} \wedge d\bar{P}_{\beta} + c.c. \Big).
\]
In this way, we recover the expressions found in~\cite[Section 4.2]{zabzine2009generalized}.
\end{example}

We now use our notion of GK potential to construct new examples of 4-dimensional GK structures for which the Hitchin Poisson structure is \emph{not of constant rank}, dropping from full rank to zero along a codimension 2 submanifold.
\begin{example}
Recall from Example~\ref{affpois} that the Poisson structure $\Pi = x \partial_{x} \wedge \partial_{y}$ on $\bb{C}^{2}$ has symplectic groupoid $\cal{G} = \bb{C}^{4}$, with source and target maps
\[
t(a,b,x,y) = (e^{a}x,y + xb), \qquad s(a,b,x,y) = (x, y),
\]
and symplectic form $\Omega = da \wedge d(y + xb) - db \wedge dx$. Viewing this as the trivial Morita equivalence we can upgrade this to an example of a GK structure by the appropriate choice of a brane bisection. In order to apply Theorem \ref{GK potential} we choose a Darboux chart for the groupoid as follows:
\[
(p_{1}, p_{2}, q_{1}, q_{2}) = (a, - b, y + xb, x).
\]
This puts the symplectic form into the standard form 
\[
\Omega = dp_{1} \wedge dq_{1} + dp_{2} \wedge dq_{2}.
\]
We take $q_{1}, q_{2}$ to be coordinates on the Lagrangian and $p_{1}, p_{2}$ the cotangent fibre coordinates. The source and target maps then have the form
\[
t(p_{1}, p_{2}, q_{1}, q_{2}) = (e^{p_{1}}q_{2}, q_{1}),  
\qquad
s(p_{1}, p_{2}, q_{1}, q_{2}) = (q_{2}, q_{1} + p_{2}q_{2}).
\]
According to Theorem~\ref{GK potential}, to define a brane it suffices to choose a real-valued function of $q_{1}$ and $q_{2}$. We take 
\[
K(q_{1}, q_{2}) = q_{1} \bar{q}_{1} + q_{2} \bar{q}_{2}.
\]
Then the brane is given by the graph of the $(1,0)$-form $-i \partial K = -i(\bar{q}_{1}dq_{1} + \bar{q}_{2}dq_{2})$. That is, the brane is given by 
\[
\cal{L} = \{(p_1,p_2,q_1,q_2)\ |\  p_1=-i \bar{q}_{1},\ \  p_2= -i \bar{q}_{2} \}.
\]
Using $q_{1}$ and $q_{2}$ as coordinates on $\cal{L}$, the source and target maps (restricted to $\cal{L}$) are 
\[
t(q_{1}, q_{2}) = (e^{-i \bar{q}_{1}}q_{2}, q_{1}), \qquad s(q_{1}, q_{2}) = (q_{2}, q_{1} - i |q_{2}|^{2}).
\]
Because these are diffeomorphisms, $\cal{L}$ is a bisection of $s$ and $t$. Hence the triple $(\cal{G}, \Omega, \cal{L})$ defines a degenerate GK structure of symplectic type. By Theorem~\ref{GK potential}, the 2-form is given by
\[
F = i \partial \bar{\partial} K = i (dq_{1} \wedge d \bar{q}_{1} + dq_{2} \wedge d \bar{q}_{2}).
\]
In these coordinates, the Hitchin Poisson structure is given by 
\[
Q = \frac{1}{2i}(q_{2} \partial_{q_{2}} \wedge \partial_{q_{1}} - \bar{q}_{2} \partial_{\bar{q}_{2}} \wedge \partial_{\bar{q}_{1}}),
\]
and the complex structures $I_{\pm}$ can be specified by their holomorphic coordinate functions: the $I_{+}$-holomorphic functions are given by $t^{\ast}x = e^{-i \bar{q}_{1}}q_{2}$ and $t^{\ast}y = q_{1}$, while the $I_{-}$-holomorphic functions are given by $s^{\ast}x = q_{2}$ and $s^{\ast}y = q_{1} - i |q_{2}|^{2}$.

In order to see that this defines a  GK structure, we must determine whether, and where, the induced metric $g$ is positive definite. First we determine where $g$ is invertible: by Lemma \ref{metric11} this happens precisely where $\omega_{-} = g I_{-} = F^{(1,1)_{-}}$ is invertible, where $F^{(1,1)_{-}}$ is the $(1,1)$-component of $F$ with respect to the $I_{-}$ complex structure. We have 
\[
F^{(1,1)_{-}} = i dq_{1} \wedge d\bar{q}_{1} + i(1 - 2|q_{2}|^{2})dq_{2}\wedge d\bar{q}_{2} + \bar{q}_{2}dq_{2} \wedge dq_{1} + q_{2} d\bar{q}_{2} \wedge d \bar{q}_{1},
\]
so that 
\[
F^{(1,1)_{-}} \wedge F^{(1,1)_{-}} = -2(1 - |q_{2}|^{2}) dq_{1} \wedge d\bar{q}_{1} \wedge dq_{2} \wedge d\bar{q}_{2}.
\]
Hence $g$ is invertible whenever $|q_{2}| \neq 1$. Along $q_2=0$, we have $s^{\ast}(dx \wedge dy) = dq_{2} \wedge dq_{1}$, so that $I_{-}$ coincides with the standard complex structure in coordinates $(q_1,q_2)$, and $F^{(1,1)_{-}} = i dq_{1} \wedge d\bar{q}_{1} + idq_{2}\wedge d\bar{q}_{2}$. Since these formulas coincide with the standard K\"ahler structure on $\bb{C}^2$, we know that $g$ is positive definite when $q_{2} = 0$ and therefore in the region $|q_{2}| < 1$. Therefore this defines a generalized K\"ahler structure on $\bb{C} \times D$, where $D = \{ q_{2} \ | \ |q_{2}| < 1 \}$. The metric has the explicit form 
\[
g = 2(dq_{1} d\bar{q}_{1} + dq_{2} d\bar{q}_{2} + i \bar{q}_{2} dq_{1} dq_{2} - i q_{2} d\bar{q}_{1} d \bar{q}_{2}).
\]
Note that if we restrict to the disc $\{q_{1} = c\} \times D$ the metric pulls back to $2dq_{2} d\bar{q}_{2}$, and so this disc has finite volume. Therefore the metric is not complete by the Hopf-Rinow theorem.
\end{example}

\begin{example}\label{gkansatz}
We now generalize the previous example by choosing a different generalized K\"ahler potential $K$, of the following form:
\[
K(q_{1}, q_{2}) = a(q_{1},\bar{q}_{1}) + b(q_{2}, \bar{q}_{2}),
\]
for real-valued functions $a$ and $b$. The brane is then given by 
\[
\cal{L} = \{(p_1,p_2,q_1,q_2)\ |\  p_1=-i \partial_{q_{1}}a, \ \ p_2 = -i \partial_{q_{2}}b \}.
\]
Again, $t$ and $s$ restrict to $\cal{L}$ to be diffeomorphisms, and hence $\cal{L}$ defines a brane bisection, defining a  degenerate GK structure. In terms of the real-valued functions
\[
\alpha(q_{1}, \bar{q}_{1}) = \frac{\partial^2 a}{\partial{q_{1}} \partial{\bar{q}_{1}}}, \qquad \beta(q_{2}, \bar{q}_{2}) = \frac{\partial^2 b}{\partial{q_{2}} \partial{\bar{q}_{2}}},
\]
we have 
\[
F = i (\alpha dq_{1} \wedge d \bar{q}_{1} + \beta dq_{2} \wedge d\bar{q}_{2}),
\]
which has $(1,1)$ component with respect to $I_{-}$ given by 
\[
F^{(1,1)_{-}} = i \alpha dq_{1} \wedge d\bar{q}_{1} + i \beta(1 - 2 \alpha \beta |q_{2}|^{2})dq_{2}\wedge d\bar{q}_{2} + \alpha \beta \bar{q}_{2}dq_{2} \wedge dq_{1} + \alpha \beta q_{2} d\bar{q}_{2} \wedge d \bar{q}_{1},
\]
so that 
\[
F^{(1,1)_{-}} \wedge F^{(1,1)_{-}} = -2 \alpha \beta (1 - \alpha \beta |q_{2}|^{2}) dq_{1} \wedge d\bar{q}_{1} \wedge dq_{2} \wedge d\bar{q}_{2}.
\]
Furthermore, the induced metric is given by 
\[
g = 2(\alpha dq_{1} d\bar{q}_{1} + \beta dq_{2} d\bar{q}_{2} + i \alpha \beta \bar{q}_{2} dq_{1} dq_{2} - i \alpha \beta q_{2} d\bar{q}_{1} d \bar{q}_{2}).
\]
We want the metric to be positive definite. Setting $q_{2} = 0$ shows that we must have $\alpha, \beta > 0$. The expression for $F^{(1,1)_{-}} \wedge F^{(1,1)_{-}}$ shows that $g$ will be invertible precisely when 
\[
1 \neq \alpha \beta |q_{2}|^{2}.
\]
Setting $q_{2} = 0$ shows that we must therefore require that $\alpha \beta |q_{2}|^{2} < 1$. Finally, because $\alpha$ and $\beta$ depend only on $q_{1},\bar{q}_{1}$ and $q_{2},\bar{q}_{2}$ respectively, $\alpha$ must be bounded by a constant and $\beta$ must be bounded by $\frac{C}{|q_{2}|^{2}}$, for $C$ a positive constant. Under these assumptions, we obtain examples of generalized K\"ahler structures of symplectic type.
\end{example}

\begin{proposition}\label{globex}
In the context of the previous example, choose the generalized K\"ahler potential to be 
\[
K(q_{1}, q_{2}) = \frac{|q_{1}|^{2}}{C} - Li_{2}(-|q_{2}|^{2}),
\]
for $C > 1$ a constant, where $Li_{2}(z) = - \int_{0}^{z} \log(1-u)\frac{du}{u}$ is the dilogarithm. Then we get $\alpha = \frac{1}{C}$ and $\beta = \frac{1}{1 + |q_{2}|^2}$, and hence the metric is given by 
\[
g = 2(\frac{1}{C} dq_{1} d\bar{q}_{1} + \frac{1}{1+|q_{2}|^2}(dq_{2} d\bar{q}_{2} + \frac{i}{C} \bar{q}_{2} dq_{1} dq_{2} - \frac{i}{C} q_{2} d\bar{q}_{1} d \bar{q}_{2})).
\]
This gives a generalized K\"ahler structure on $\bb{R}^4$ for which the metric is complete.
\begin{proof}
After the argument of Example~\ref{gkansatz}, we need only show the completeness of the metric.  Note that the metric is translation invariant in the $q_{1}$ direction. Quotienting by a $\bb{Z}^2$ lattice, we obtain a metric on $S^1 \times S^1 \times \bb{C}$. Completeness of this metric is equivalent to completeness of the original.  To show completeness, we need only investigate geodesics escaping to infinity in the $q_{2}$-direction. In coordinates $q_{1} = x + iy$ and $q_{2} = \sinh(s)e^{i \theta}$, the metric has the following form:
\[
g = \frac{2}{C}(dx^2 + dy^2) + 2(ds^2 + \tanh^2(s)d\theta^2 - \frac{2}{C}\tanh^2(s) dx d\theta - \frac{2}{C}\tanh(s)dyds).
\]
For any geodesic $\gamma(t) = (x(t), y(t), s(t), \theta(t))$, we have 
\[
\sqrt{g(\dot{\gamma}, \dot{\gamma})} \geq \sqrt{\frac{2 (C-1)}{C}} \dot{s},
\]
from which it follows that the length of the curve over the interval $[t_{0}, t_{1}]$ is bounded below in the following way:
\[
L(\gamma) \geq \sqrt{\frac{2 (C-1)}{C}} (s(t_{1}) - s(t_{0})).
\]
Therefore the length of a curve that escapes to infinity is unbounded and so the metric is complete.

\end{proof}
\end{proposition}

\section{The Picard group}\label{picsec}
In this section we focus on generalized K\"ahler structures where the two holomorphic Poisson structures are isomorphic. In this case, by choosing an isomorphism, we can assume that the two holomorphic Poisson structures coincide and then study the self-Morita equivalences of the given Poisson structure. This leads to the notion of the \emph{Picard group} of a Poisson structure $(X, \sigma)$, first introduced by Weinstein and Bursztyn in the smooth category in \cite{bursztyn2004picard}. Since there is a composition for Morita equivalences, (integrable) holomorphic Poisson manifolds can be viewed as the objects of a groupoid where the morphisms are given by holomorphic symplectic Morita equivalences.
\begin{definition}
The \emph{holomorphic Picard groupoid} $\PG$ is the category whose objects are integrable holomorphic Poisson manifolds and whose morphisms are isomorphism classes of holomorphic symplectic Morita equivalences. The \emph{Picard group} of a holomorphic Poisson manifold $(X, \sigma)$ is the automorphism group of $(X, \sigma)$ in $\PG$: 
\[
\Pic(X, \sigma) = \Hom[\PG]{ (X, \sigma), (X, \sigma) }.
\]
\end{definition}
Since Morita equivalences equipped with bisections also compose, it is possible to upgrade the above groupoid so that the morphisms are holomorphic symplectic Morita equivalences with brane bisection.
\begin{definition} 
The \emph{Picard groupoid with branes} $\BPG$  
is the category whose objects are integrable holomorphic Poisson manifolds and whose morphisms are isomorphism classes of holomorphic symplectic Morita equivalences equipped with brane bisections. The automorphism group of the object $(X, \sigma)$ is denoted by 
\[
\BPic(X, \sigma) = \Hom[\BPG]{ (X, \sigma), (X, \sigma) }.
\]
\end{definition}
\begin{remark}
Morphisms between two fixed Poisson manifolds $\Hom[\BPG]{ (X_{+}, \sigma_{+}), (X_{-}, \sigma_{-})}$ correspond, by Theorem~\ref{main}, to degenerate GK structures of symplectic type.  
\end{remark}

There is a natural forgetful functor from $\BPG$ to $\PG$ which drops the data of the brane bisection. This gives rise to a homomorphism 
\[
\BPic(X, \sigma) \to \Pic(X, \sigma),\qquad [(Z, \Omega, \lambda)] \mapsto [(Z, \Omega)].
\]
The kernel of this map consists of objects where the underlying Morita equivalence is trivial, i.e. isomorphic to the Weinstein groupoid $(\Sigma(X), \Omega)$. Therefore, the kernel is given by the image of the following natural homomorphism 
\[
\mathrm{Bis}^\cal{L}(\Sigma(X)) \to \BPic(X, \sigma), \qquad \lambda \mapsto [(\Sigma(X), \Omega, \lambda)],
\]
where $\mathrm{Bis}^\cal{L}(\Sigma(X))$ is the group of brane bisections in $(\Sigma(X), \Omega)$. Note that we are viewing a brane bisection as a map $\lambda: X \to Z$ such that $\pi_{-} \circ \lambda = id_{X}$, the map $\phi_{\lambda} := \pi_{+} \circ \lambda$ is a diffeomorphism, and $\lambda^{\ast} \Im(\Omega) = 0$.

Using Theorem~\ref{main}, we are able to give a concrete description of $\BPic(X, \sigma)$, as follows. An object $[(Z, \Omega, \lambda)]$ of this group is a Morita self-equivalence with brane bisection of the holomorphic Poisson structure $(X, \sigma)$, where we let $M$ denote the underlying smooth manifold of $X$, $I$ its  complex structure, and $Q = -4 \Im(\sigma)$.  By Theorem~\ref{main}, this self-equivalence with brane corresponds to a degenerate GK structure of symplectic type, i.e. a solution $(I_{+}, I_{-}, Q, F)$ of equations \ref{star1} and \ref{star2}.  Since we are viewing the brane as a section $\lambda$ of $\pi_{-}$ in this correspondence, we obtain 
$I_{-} = I$, $F = \lambda^{\ast} \Omega$, and $I_{+} = (\phi_{\lambda}^{-1})_{\ast}(I)$, where $\phi_{\lambda} = \pi_{+} \circ \lambda$. Therefore, given the holomorphic Poisson structure, the remaining data is encoded by the real 2-form $F$ and the diffeomorphism $\phi_{\lambda}$. Using equation \ref{star1} to express $I_{+}$ as $I^{F} := I + QF$, equation \ref{star2} then becomes 
\[
FI + I^{*}F + FQF = 0,
\]
and the relation between $\phi_{\lambda}$, $I_{+}$ and $I_{-}$ can then be expressed as 
\[
(\phi_{\lambda})_{\ast}(I^{F}) = I.
\]
%
%
%
Based on these observations, we define the following subgroup of $\mathrm{Diff}_Q(M) \ltimes \Omega^{2,\mathrm{cl}}(M)$, the semi-direct product of the group of diffeomorphisms preserving $Q$ with the group of closed 2-forms: 
\[
\mathrm{Aut}_\cal{C}({I, \sigma}) = \{ (\phi, F) \in \mathrm{Diff}_Q(M) \times \Omega^{2,\mathrm{cl}}(M) \ | \ FI + I^{\ast} F + FQF = 0\text{ and } \phi_{\ast}(I^{F}) = I \},
\]
with multiplication defined as follows:
\begin{equation}  \label{groupproduct}
(\phi_{1}, F_{1}) \ast (\phi_{2}, F_{2}) = (\phi_{1} \circ \phi_{2}, \phi_{2}^{\ast}F_{1} + F_{2}).
\end{equation}
This is the group of \emph{Courant automorphisms of $(I, \sigma)$} when the holomorphic Poisson structure is viewed as a generalized complex structure (See~\cite{gualtieri2010branes} for a study of this group). An upshot of the present discussion is the existence of a map between $\BPic(X, \sigma)$ and $\mathrm{Aut}_\cal{C}({I, \sigma})$, and a consequence of Theorem~\ref{main} is the fact that these two groups are isomorphic.
\begin{corollary}
There is an isomorphism of groups 
\[
\chi : \BPic(X, \sigma) \to \mathrm{Aut}_\cal{C}({I, \sigma}), \qquad [Z, {\Omega},\lambda] \mapsto (\pi_{+} \circ \lambda, \lambda^{\ast} {\Omega} ).
\]
\begin{proof} The above discussion shows that the map $\chi$ is well-defined. Hence it remains to show that $\chi$ is a homomorphism and a bijection. 

Step 1: $\chi$ is a group homormophism. Let $(Z_{1}, \Omega_{1}, \lambda_{1})$ and $(Z_{2}, \Omega_{2}, \lambda_{2})$ be Morita equivalences with brane bisections. The product of the bimodules is given by the quotient 
\[
Z_{1} \ast Z_{2} = (Z_{1} \times_{X} Z_{2}) / \Sigma(X),
\]
with symplectic form $\Omega$ defined via symplectic reduction, and the product of the bisections is given by 
\[
\lambda_{1} \ast \lambda_{2}(x) = [\lambda_{1} \circ \pi_{+,2} \circ \lambda_{2}(x), \lambda_{2}(x) ].
\]
Therefore $\pi_{+} \circ (\lambda_{1} \ast \lambda_{2}) = (\pi_{+,1} \circ \lambda_{1}) \circ (\pi_{+,2} \circ \lambda_{2})$, and 
\begin{align*}
(\lambda_{1} \ast \lambda_{2})^{\ast} \Omega &= (\lambda_{1} \circ \pi_{+,2} \circ \lambda_{2})^{\ast} \Omega_{1} + \lambda_{2}^{\ast} \Omega_{2} \\
 &= (\pi_{+,2} \circ \lambda_{2})^{\ast} (\lambda_{1}^{\ast} \Omega_{1}) + \lambda_{2}^{\ast} \Omega_{2},
\end{align*}
which, according to equation \ref{groupproduct}, is the product of $(\pi_{+, 1} \circ \lambda_{1}, \lambda_{1}^{\ast} {\Omega_{1}} )$ and $(\pi_{+, 2} \circ \lambda_{2}, \lambda_{2}^{\ast} {\Omega_{2}} )$.

Step 2: $\chi$ is an isomorphism. We show this by defining an explicit inverse $\zeta : \mathrm{Aut}_\cal{C}({I, \sigma}) \to \BPic(X,\sigma)$. Given $(\phi, F) \in \mathrm{Aut}_\cal{C}({I, \sigma})$, Proposition~\ref{Star->ME} implies that we get the following element of the Picard 
\[
\xymatrix{
&(\Sigma(X),\Omega + t^{\ast}F)\ar[dl]_-{\phi \circ t}\ar[dr]^-{s} & \\
(X,\sigma) & &(X,\sigma) 
}
\]
where the brane bisection is given by the identity bisection $\epsilon$. We therefore take this to define $\zeta(\phi, F)$. Theorem \ref{main} then implies that $\zeta$ and $\chi$ are inverse to each other. 
\end{proof}
\end{corollary}
As a result of this isomorphism, we immediately see that the kernel of the map $\mathrm{Bis}^\cal{L}(\Sigma(X)) \to \BPic(X, \sigma)$ is given by the subgroup $\mathrm{IsoLBis}(\Sigma(X))$ of holomorphic Lagrangian bisections of $(\Sigma(X), \Omega)$ which induce the identity diffeomorphism on $M$. Collecting these facts we get the following exact sequence of groups 
\begin{equation} \label{exactseq}
1 \to \mathrm{IsoLBis}(\Sigma(X)) \to \mathrm{Bis}^\cal{L}(\Sigma(X)) \to \mathrm{Aut}_\cal{C}({I, \sigma}) \to \Pic(X, \sigma).
\end{equation}
\begin{remark} Bursztyn and Fernandes studied a similar sequence in \cite{bursztyn2015picard} in the setting of real smooth Poisson structures. The above may be seen as a generalization of their results to the case of holomorphic Poisson structures. 
\end{remark}
\begin{remark}
A central claim of this paper has been that the holomorphic substructure underlying GK manifolds of symplectic type consists of holomorphic symplectic Morita equivalences, and that the additional real data needed to determine the metric consists of a brane bisection. This is reflected in the sequence \ref{exactseq} by the fact that $\mathrm{Aut}_\cal{C}({I, \sigma})$ is an extension of its image in the holomorphic Picard group by the group of brane bisections.
\end{remark}

\section{Generalized K\"ahler metrics via Hamiltonian flows}\label{hamflo}
An important application of the ideas of Section~\ref{picsec} is to the construction and deformation of generalized K\"ahler metrics.  The basic idea is as follows: 
Theorem \ref{main} allows us to view GK structures of symplectic type as Morita equivalences with brane bisection, which are morphisms in the groupoid $\BPG$. So, it is possible to compose them.  More precisely, if $(Z_{1}, \Omega_{1}, \cal{L}_{1})$ and $(Z_{2}, \Omega_{2}, \cal{L}_{2})$ are degenerate GK structures of symplectic type going between holomorphic Poisson structures $(X_{1}, \sigma_{1}), (X_{2}, \sigma_{2})$ and $(X_{2}, \sigma_{2}), (X_{3}, \sigma_{3})$ respectively, then we may compose them to get a degenerate GK structure going between $(X_{1}, \sigma_{1})$ and $(X_{3}, \sigma_{3})$. In particular, there is an action of the group $\BPic(X_{+}, \sigma_{+})$ on the space of morphisms $\BPG((X_{+}, \sigma_{+}), (X_{-}, \sigma_{-}))$ and we can use this to \emph{deform} GK structures of symplectic type. 
Indeed, all of the constructions of GK metrics contained in~\cite{MR2371181,MR2217300,gualtieri2010branes} are special cases of this.  

The idea of the construction in \cite{gualtieri2010branes} is to start with the infinitesimal counterpart of $\BPic(X,\sigma) \cong \mathrm{Aut}_\cal{C}({I, \sigma})$, i.e. an infinitesimal Courant symmetry, integrate it to a Courant automorphism (a path in $\mathrm{Aut}_\cal{C}({I, \sigma})$), and then use this to deform a given GK structure. We will see how this manifests itself in terms of Morita equivalences with brane bisections. In particular we will see that in certain cases these deformations are obtained by flowing the brane bisections using Hamiltonian vector fields.
%

The Lie algebra of $\mathrm{Aut}_\cal{C}({I, \sigma})$ is the following subalgebra of $\frak{X}_{Q}(M) \ltimes \Omega^{2, cl}(M)$, the semi-direct product of the Lie algebra of vector fields preserving $Q$ with the group of closed $2$-forms:
\[
\frak{aut}_\cal{C}(I, \sigma) = \{ (V, \omega) \in \frak{X}_{Q}(M) \times \Omega^{2, cl}(M) \ | \ \omega I + I^{\ast} \omega = 0 \text{ and } \cal{L}_{V}I = Q \omega \}.
\]
This is the Lie algebra of \emph{infinitesimal Courant symmetries} of the holomorphic Poisson structure $(I, \sigma)$. Let $(V, \omega) \in \frak{aut}_\cal{C}(I, \sigma)$ be an element of this Lie algebra. This defines an infinitesimal automorphism of the Courant algebroid $T \oplus T^{*}$ and therefore it integrates to the following family of automorphisms of $T\oplus T^{*}$:
\[
(\phi_{t}, F_{t} = \int_{0}^{t} (\phi_{s}^{\ast}\omega) ds),
\]
where $\phi_{t}$ is the flow of the vector field $V$. In \cite[Section 7]{gualtieri2010branes} it is shown that this family lies in $\mathrm{Aut}_\cal{C}({I, \sigma})$ for all $t$ (where it is defined), and therefore this family defines the $1$-parameter subgroup integrating $(V, \omega)$. Hence we have the following result.
\begin{lemma}
The exponential map 
\[
\exp : \frak{aut}_\cal{C}(I, \sigma) \to \mathrm{Aut}_\cal{C}({I, \sigma}) \cong \BPic(X, \sigma)
\]
is given by flowing the above family of automorphisms to $t = 1$:
\[
(V, \omega) \mapsto (\phi_{1}, F_{1} = \int_{0}^{1} (\phi_{s}^{\ast} \omega) ds).
\]
\end{lemma}

The Lie algebra of the group of brane bisections $\mathrm{Bis}^\cal{L}(\Sigma(X))$ is given by the space of closed $1$-forms $\Omega^{1, cl}(M)$ with Lie bracket induced by $Q$ \cite{xu1997flux}. The exponential map $\exp : \Omega^{1, cl}(M) \to \mathrm{Bis}^\cal{L}(\Sigma(X))$ has the following description: given $\alpha \in \Omega^{1, cl}(M)$, we can consider the vector field induced by $t^{\ast}\alpha$ on the Weinstein groupoid with respect to the imaginary part of the symplectic form $\omega = \Im(\Omega)$:
\[
V_{t^{\ast} \alpha} = \omega^{-1}(t^{\ast} \alpha).
\]
This is the right-invariant vector field associated to $\alpha$ when it is viewed as a section of the Lie algebroid $T^{\ast}_{Q}M$ of $(\Sigma(X), \omega)$. It is $t$-related to the vector field $Q(\alpha)$ on $X$. Let $\phi_{t}$ be the flow of $V_{t^{\ast}\alpha}$. Then the $1$-parameter subgroup integrating $\alpha$ is given by the family $\lambda_{t} = \phi_{t} \circ \epsilon$ of bisections, where $\epsilon$ is the identity bisection, so that  
\[
\exp(\alpha) = \lambda_{1}.
\]
Mapping the $1$-parameter subgroup $\lambda_{t}$ to $\mathrm{Aut}_\cal{C}({I, \sigma})$ and taking the derivative at $t = 0$ determines the induced Lie algebra morphism $\Omega^{1, cl}(M) \to \frak{aut}_\cal{C}(I, \sigma)$: $\lambda_{t} \in \mathrm{Bis}^\cal{L}(\Sigma(X))$ gets sent to the family $[ \Sigma(X), \Omega, \lambda_{t} ] \in \BPic(X, \sigma)$ of Morita equivalences with brane bisections, and this corresponds in $\mathrm{Aut}_\cal{C}({I, \sigma})$ to the family of Courant automorphisms given by $(\psi_{t}, \int_{0}^{t} (\psi_{s}^{\ast} d^{c} \alpha) ds )$, where $\psi_{t}$ is the flow of $Q(\alpha)$ and $d^{c} = i (\bar{\del} - \del)$. Therefore the map of Lie algebras is given by 
\[
\Omega^{1, cl}(M) \to \frak{aut}_\cal{C}(I, \sigma), \qquad \alpha \mapsto (Q(\alpha), d^{c} \alpha). 
\]
One upshot of the present discussion is that the exponential map for $\mathrm{Aut}_\cal{C}({I, \sigma})$ has a particularly nice description when it is applied to elements in the image of $\Omega^{1, cl}(M)$.
\begin{proposition} \label{flow1}
Let $\alpha \in \Omega^{1, cl}(M)$ be a real closed $1$-form and let $(Q(\alpha), d^{c} \alpha) \in  \frak{aut}_\cal{C}(I, \sigma)$ be the infinitesimal Courant symmetry that it determines. Exponentiating this symmetry yields a family of holomorphic symplectic Morita equivalences with brane bisections. This family has the following simple form:
\[
[ (\Sigma(X), \Omega, \phi_{t} \circ \epsilon )],
\]
where $(\Sigma(X), \Omega)$ is the Weinstein groupoid viewed as a trivial Morita equivalence, $\phi_{t}$ is the flow of the vector field $(\Im \Omega)^{-1}(t^{\ast} \alpha)$ on the groupoid, and $\phi_{t} \circ \epsilon$ is the result of applying this flow to the identity bisection. 
\end{proposition}

We now explain how the flow of infinitesimal Courant symmetries can be used to deform a given GK structure of symplectic type to a nearby one. Such deformations are obtained via the action of the group $\BPic(X_{+}, \sigma_{+})$ on the space of morphisms $\BPG((X_{+}, \sigma_{+}), (X_{-}, \sigma_{-}))$. More precisely, let
\[
\xymatrix{
 &(Z,\Omega, \cal{L})\ar[dl]_-{\pi_{+}}\ar[dr]^-{\pi_{-}} & \\
 (X_{+},\sigma_{+})& & (X_{-}, \sigma_{-})
}
%
%
\]
be a GK structure viewed as a Morita equivalence with brane bisection, and let $\zeta(\phi,F)$
be an element of $\BPic(X_{+}, \sigma_{+})$ corresponding to the Courant automorphism $(\phi, F) \in \mathrm{Aut}_\cal{C}({I_{+}, \sigma_{+}})$. Composing with the above Morita equivalences, we obtain $\zeta(\phi,F) \ast (Z,\Omega,\cal{L}) = (Z,\Omega + \pi_{+}^{\ast}F, \cal{L})$, again a morphism in $\BPG((X_{+}, \sigma_{+}), (X_{-}, \sigma_{-}))$ as depicted below. 
\[
\xymatrix{
 &(Z,\Omega + \pi_{+}^{\ast}F, \cal{L})\ar[dl]_-{\phi\circ \pi_{+}}\ar[dr]^-{\pi_{-}} & \\
(X_{+},\sigma_{+})& & (X_{-}, \sigma_{-})
}
\]
Since the positivity of the induced metric is an open condition, this will define a new GK structure if $(\phi, F)$ is `close enough' to the identity. Given an infinitesimal symmetry $(V, \omega) \in \frak{aut}_\cal{C}(I_{+}, \sigma_{+})$ we can exponentiate it to the family $\zeta(\exp(t(V,\omega)))$ in $\BPic(X_{+}, \sigma_{+})$. Then $\zeta(\exp(t(V,\omega))\ast (Z,\Omega,\cal{L})$
defines a family of degenerate GK structures deforming the initial structure, and the metric will remain positive-definite for sufficiently small $t$. This family of deformations has a particularly nice form for infinitesimal symmetries in the image of $\Omega^{1, cl}(M)$.

\begin{proposition} \label{flow2}
Let $(Z,\Omega, \cal{L})$ be a GK structure of symplectic type, viewed as a holomorphic symplectic Morita equivalence with brane bisection going between the holomorphic Poisson structures $(I_{\pm}, \sigma_{\pm})$. Let $\alpha \in \Omega^{1, cl}(M)$ be a real closed $1$-form, and let $(Q(\alpha), d^{c} \alpha) \in  \frak{aut}_\cal{C}(I_{+}, \sigma_{+})$ be the infinitesimal Courant symmetry that it determines. Exponentiating this symmetry, we get a 1-parameter family of Courant automorphisms which act on $(Z,\Omega, \cal{L})$ to produce a family of degenerate GK structures of symplectic type deforming the given one. This family has the following simple form:
\[
[(Z, \Omega, \cal{L}_{t} = \eta_{t}(\cal{L}))],
\]
where $\eta_{t}$ is the flow of the vector field $V_{\pi_{+}^{\ast} \alpha} = (\Im \Omega)^{-1}( \pi_{+}^{\ast} \alpha)$, and $\cal{L}_{t} = \eta_{t}(\cal{L})$ is the result of applying this flow to the brane bisection $\cal{L}$.
\begin{proof}
Let $\lambda_{t} = \phi_{t} \circ \epsilon$ be the 1-parameter subgroup corresponding to $\alpha$ in $\mathrm{Bis}^\cal{L}(\Sigma(X))$, where $\phi_{t}$ is the flow of $V_{t^{\ast} \alpha}$. By Proposition~\ref{flow1}, the family of degenerate GK structures obtained by exponentiating $(Q(\alpha), d^{c} \alpha)$ and acting on the given GK structure is given by the composition of $(\Sigma(X), \Omega, \lambda_{t})$ and $ (Z,\Omega, \cal{L})$. This gives $(Z, \Omega, \cal{L}_{t} = \lambda_{t} \ast \cal{L})$, where 
\[
(\lambda_{t} \ast \cal{L})(x) = \lambda_{t}(\pi_{+} \circ \cal{L}(x)) \ast \cal{L}(x),
\]
and where we are using the action $\ast$ of $\Sigma(X)$ on $Z$. Note that we are abusing notation by using $\cal{L}$ to denote both the section of $\pi_{-}$ and its image in $Z$. Now let $V_{\pi_{+}^{\ast} \alpha} = \Im(\Omega)^{-1}( \pi_{+}^{\ast} \alpha)$, a vector field on $Z$, and let $\eta_{t}$ be its flow. We claim that $\cal{L}_{t} = \eta_{t}(\cal{L})$. This can be seen as follows. Using the brane $\cal{L}$ we define a smooth symplectomorphism 
\[
\Phi: (\Sigma(X_{+}), \Im(\Omega)) \to (Z, \Im(\Omega)), \qquad g \mapsto g \ast (r \circ s)(g),
\]
where $r : X_{+} \to Z$ is the section of $\pi_{+}$ induced by $\cal{L}$, and $s$ is the source map of the Weinstein groupoid. This map satisfies the relation $\pi_{+} \circ \Phi = t$, and therefore $d\Phi (V_{t^{\ast} \alpha}) = V_{\pi_{+}^{\ast} \alpha}$. This implies that $\Phi$ intertwines the flows $\phi_{t}$ and $\eta_{t}$ of the vector fields $V_{t^{\ast} \alpha}$ and $V_{\pi_{+}^{\ast} \alpha}$, respectively. Then, by the definitions of $\Phi$ and $\cal{L}_{t}$ we see that on the one-hand 
\[
\Phi \circ \lambda_{t} \circ \pi_{+}\circ \cal{L}(x) = \lambda_{t}(\pi_{+} \circ \cal{L}(x)) \ast \cal{L}(x) = \cal{L}_{t}(x),
\]
and by the property that $\Phi$ intertwines $\phi_{t}$ and $\eta_{t}$ we see that on the other hand 
\[
\Phi \circ \lambda_{t} \circ \pi_{+}\circ \cal{L}(x) = \eta_{t} \circ \Phi(1_{\pi_{+} \circ \cal{L}(x)}) = \eta_{t}(\cal{L}(x)).
\]
\end{proof}
\end{proposition}

\begin{remark}
In Propositions \ref{flow1} and \ref{flow2}, we may specialize to the case of exact $1$-forms $\alpha = dK$, in which case the families we obtain are given by flowing the brane bisections by Hamiltonian vector fields.
\end{remark}

\begin{example}
Recall from Example~\ref{hyperKahler} that the Morita equivalence of a hyper-K\"ahler structure $(M, I, J, K, g)$ is given by 
\[
(Z, \Omega) = (X_{+}, \Omega_{+}) \times (X_{-}, -\Omega_{-}),
\]
with brane bisection $\cal{L}$ given by the diagonally embedded copy of $M$. It was observed in~\cite{MR1702248} that such a GK structure may be deformed to a new GK structure which is not hyper-K\"ahler using a real-valued function $f$. 
In the present setting we can view this as a deformation obtained by flowing the brane bisection using a Hamiltonian vector field of $f$. More precisely, given a real valued-function $f$, Proposition \ref{flow2} says that the brane is flowed by the Hamiltonian vector field of $\pi_{+}^{\ast} f$, using the imaginary part of the symplectic form on $Z$, namely, $\Im (\Omega_{+}, - \Omega_{-}) = (\omega_{K}, - \omega_{K})$. Therefore, the brane is given by the flow of 
\[
(\omega_{K}^{-1}(df), 0),
\]
and so if $\phi_{t}$ is the flow of the Hamiltonian vector field of $f$ with respect to $\omega_{K}$ then the deformed brane is given by 
\[
\cal{L}_{t} = \{ (\phi_{t}(m), m) \ | \ m \in M\},
\]
which is the graph of $\phi_{t}$.
\end{example}

\section{Universal local construction via time-dependent flows}\label{locdef}
Let $(M, I, \sigma)$ be a holomorphic Poisson structure, which determines a degenerate GK structure as in Example~\ref{trivdeg}.  As explained in Example~\ref{integtriv}, the corresponding Morita equivalence provided by Theorem~\ref{main} is the symplectic groupoid $(\Sigma(X), \Omega, \epsilon)$ integrating $\sigma$, with the identity bisection $\eps$ as its brane bisection. Propositions \ref{flow1} and \ref{flow2} tell us that given a real-valued function $f$ we can construct a family of degenerate GK structures deforming the given one by deforming the identity bisection $\epsilon$ using the Hamiltonian vector field of $t^{\ast}f$ on the groupoid: $V_{t^{\ast}df} = \omega^{-1}(dt^{\ast}f)$, where here we decompose $\Omega = B+i\omega$ into its real and imaginary parts. In this section, we show that, locally, all degenerate GK structures of symplectic type arise in this way if we allow the function $f$ to depend on time. This gives a universal local construction (similar to the one using a generalized K\"ahler potential) which places a greater emphasis on the underlying holomorphic Poisson geometry. In effect, it tells us that locally a GK structure of symplectic type is determined by a holomorphic Poisson structure and a single time-dependent real-valued function.

Let $(Z, \Omega, \cal{L})$ be a GK structure viewed as a Morita equivalence with brane bisection, and let $z \in \cal{L}$ be a point on the brane. Choose a local holomorphic Lagrangian bisection $\Lambda$ passing through the point $z$ and let $U := \pi_{-}(\Lambda) \subseteq X_{-}$, and $V := \pi_{+}(\Lambda) \subseteq X_{+}$. Note first that this bisection induces a holomorphic Poisson isomorphism $\phi : (U, \sigma_{-}) \to (V, \sigma_{+})$; this means that the two holomorphic Poisson structures underlying a GK structure of symplectic type are always locally isomorphic, although not in a canonical way. Now consider the local Morita equivalence $\pi_{+}^{-1}(V) \cap \pi_{-}^{-1}(U)$ going between $(U, \sigma_{-})$ and $(V, \sigma_{+})$. Using the Lagrangian $\Lambda$, we can identify this Morita equivalence with the trivial Morita equivalence $t^{-1}(U) \cap s^{-1}(U) \subseteq \Sigma(X_{-})$. That is to say, we have a holomorphic symplectomorphism 
\begin{equation}\label{grpcht}
\Phi : t^{-1}(U) \cap s^{-1}(U) \to \pi_{+}^{-1}(V) \cap \pi_{-}^{-1}(U), \qquad g \mapsto \Lambda(t(g)) \ast g,
\end{equation}
satisfying $\pi_{+} \circ \Phi = \phi \circ t$, as well as $\pi_{-} \circ \Phi = s$ and $\Phi \circ \epsilon = \Lambda.$
 We view this as a \emph{groupoid chart}, in the sense that it identifies a neighbourhood of $z$ in $Z$ with a neighbourhood of the zero section in $\Sigma(X_-)$.  The chart is adapted to the groupoid structure, unlike the Darboux charts considered in section \ref{secpot}. In this chart, the brane $\cal{L}$ intersects the identity bisection at the point $z$. Our goal is to describe $\cal{L}$ as a Hamiltonian flow applied to the zero section: for this, we require a family of brane bisections $\cal{L}_{t}$ interpolating between $\cal{L}$ and $\Lambda$. First we show that such a family exists.

\begin{proposition} \label{bisectionpath}
There is a local family of brane bisections interpolating between a given brane $\cal{L}$ and a holomorphic Lagrangian bisection $\Lambda$.
\end{proposition}
Such a family of brane bisections consists of a family of Lagrangian submanifolds for $\omega = \Im(\Omega)$ which is transverse to both $K_{\pm} = \ker (d \pi_{\pm})$ at all times. Note that since $K_{+}$ and $K_{-}$ are symplectic orthogonal, a Lagrangian $L$ which is transverse to $K_{+}$ is automatically transverse to $K_{-}$:
\[
L \cap K_{-} = L^{\omega} \cap K_{+}^{\omega} = (L + K_{+})^{\omega} = TZ^{\omega} = 0.
\]
The linear version of this problem has an immediate solution: let $(V, \omega)$ be a $2n$-dimensional (real) symplectic vector space and let $K$ be an arbitrary $n$-dimensional subspace. Let $M_{V,K}$ denote the space of Lagrangians in $V$ which are transverse to $K$; this is a connected open subset of the Lagrangian Grassmannian, showing that the linear version of such an interpolation is available. 

\begin{proof}[Proof of Proposition \ref{bisectionpath}]
Choose a holomorphic Darboux chart centred at the point $z$ : $(TZ_{z}, \Omega_{z}) \cong (Z, \Omega)$, and let $\Lambda$ be a holomorphic Lagrangian subspace of $TZ_{z}$ which is transverse to $K_{\pm}$. This defines a (local) holomorphic Lagrangian bisection $\Lambda$. The tangent space $L = T_{z} \cal{L}$ defines a Lagrangian subspace of $(TZ_{z}, \omega_{z})$ which is also transverse to $K_{\pm}$. 
Since $M_{TZ_{z},K_+}$ is connected, we choose a family $L_{t}$ of Lagrangian subspaces of $(TZ_{z}, \omega_{z})$ which remain transverse to $K_{\pm}$ for all time interpolating between $\Lambda$ and $L$. We can view this as a family of brane bisections going from $\Lambda$ to the brane $L$. Hence it remains to find a path going from $L$ to $\cal{L}$. For this choose a Weinstein neighbourhood of $L$: $(T^{\ast}L, \Omega_0) \cong (Z, \omega)$. In this chart, $\cal{L}$ is given by the graph of a closed $1$-form $\alpha \in \Omega^{1}(L)$. Then $\cal{L}_{t} = Gr(t \alpha)$ defines a family of Lagrangians interpolating between $L$ and $\cal{L}$. Since $T\cal{L}_{z} = TL_{z}$ it follows that $T(\cal{L}_{t})_{z} = TL_{z}$ for all $t$ implying that this family is transverse to $K_{\pm}$ at all times. Combining the two families, we obtain a family of branes interpolating between the holomorphic Lagrangian $\Lambda$ and the brane bisection $\cal{L}$. This family fixes the point $z$, and at all times the tangent space at $z$ is transverse to $K_{\pm}$. Thus we obtain the required interpolation on a (possibly smaller) neighbourhood of $z$.
\end{proof}

Having the interpolating family of branes, we wish to describe it as a Hamiltonian flow. For this, we return to the groupoid chart $t^{-1}(U) \cap s^{-1}(U)$ where we have the following data:  
\begin{enumerate}
\item A holomorphic symplectic groupoid $(\Sigma(U), \Omega_{-})$ integrating $(U, \sigma_{-})$, with underlying imaginary part the smooth real symplectic groupoid $(\Sigma(U), \omega)$;
\item Over a fixed neighbourhood of $z$, $W \subseteq U$, we have a family of Lagrangian bisections of $(\Sigma(U), \omega)$, viewed as sections of the source $s$,
\[
\lambda_{t} : W \to (\Sigma(U), \omega),
\]
such that $\lambda_{0}$ is the identity bisection and $\lambda_{1}$ is the given brane bisection $\cal{L}$ viewed in the groupoid chart. Note that $\lambda_{t}(z) = z$ for all time $t$.
\end{enumerate}
Now let $\psi_{t} := t \circ \lambda_{t}$ denote the resulting family of Poisson diffeomorphisms (for the Poisson structure $Q$), and let $W_{t} = \psi_{t}(W)$. Left multiplication by the bisection defines the following family of symplectomorphisms:
\[
\tau_{t} : t^{-1}(W) \to t^{-1}(W_{t}), \ g \mapsto \lambda_{t}(t(g)) \ast g,
\]
which satisfy $s \circ \tau_{t} = s$ and  $t \circ \tau_{t} = \psi_{t} \circ t$, as well as $\tau_{t} \circ \epsilon = \lambda_{t}$. This family of symplectomorphisms defines a time-dependent vector field $Y_{t} \in \mathcal{X}^{1}(t^{-1}(W_{t}))$ via the equation
\[
Y_{t}(\tau_{t}(g)) = \frac{d}{dt} \tau_{t}(g).
\]
As explained in \cite{xu1997flux} this family of vector fields is Hamiltonian for a $t$-basic closed $1$-form:
\[
\iota_{Y_{t}}\omega = t^{\ast}(\alpha_{t}),
\]
where $\alpha_{t} \in \Omega^{1}(W_{t})$ is a closed time-dependent 1-form. In fact, the restriction of $Y_{t}$ to the bisection gives the following vector field along $\lambda_{t}$:
\[
X_{t}(x) := Y_{t}(\lambda_{t}(x)) = \frac{d}{dt} \lambda_{t}(x).
\]
Since $\lambda_{t}$ is a section of $s$, $X_{t} \in \ker(ds)$ and hence $\omega(X_{t})$ is in the image of $t^{\ast}$, defining the form $\alpha_{t}$. Note that since $t$ is a Poisson map we have 
\[
t_{\ast}(Y_{t}) = Q(\alpha_{t}),
\]
and so $\psi_{t}$ is the flow of this Hamiltonian vector field. The time-dependent form $\alpha_{t} : [0,1] \to \Omega_{cl}^{1}$ is the infinitesimal version of the family of branes $\lambda_{t} : [0,1] \to \mathrm{Bis}^\cal{L}(\Sigma)$. 

Now choose a neighbourhood $W'$ of $z$ such that $W' \subseteq W_{t}$ for all time $t$. Restricting to this neighbourhood, we have $\alpha_{t} \in \Omega^{1,cl}(W')$, and if we assume that $W'$ is contractible then $\alpha_{t}$ are exact. Choose a primitive: let $f_{t} \in C^{\infty}(W')$ be a time-dependent function such that $df_{t} = \alpha_{t}$. We conclude that it is now possible to describe the GK structure purely in terms of this function and the holomorphic Poisson structure $(I_{-}, \sigma_{-})$:

\begin{theorem}  \label{GKlocallyflow}
Let $(Z, \Omega, \cal{L})$ be a Generalized K\"ahler structure of symplectic type, viewed as a holomorphic symplectic Morita equivalence with brane bisection going between holomorphic Poisson structures $(X_{\pm}, \sigma_{\pm})$ with common imaginary part $-\frac{1}{4}Q$. Let $z \in \cal{L}$ be a chosen point on the brane. Then 
\begin{enumerate}
\item It is possible to choose a family of local brane bisections $\cal{L}_{t}$ such that $\cal{L}_{1} = \cal{L}$, $\cal{L}_{0} = \Lambda$, a holomorphic Lagrangian bisection, and such that $z \in \cal{L}_{t}$ for all $t\in[0,1]$. 
\item There is a (locally defined) time-dependent real-valued function $f_{t} \in C^{\infty}(X_{-}, \bb{R})$ such that in a neighbourhood of $z$ the family of brane bisections $\cal{L}_{t}$ is given by $\tilde{\tau}_{t}(\Lambda)$, where $\tilde{\tau}_{t}$ is the flow of the Hamiltonian vector field of $(\phi^{-1} \pi_{+})^{*} f_{t}$ with respect to the symplectic form $\Im (\Omega)$.
\item In a neighbourhood of $\pi_{-}(z)$ the GK structure is given by the data $(I_{+}, I_{-}, Q, F)$, where $I_{-}$ is the given complex structure on $X_{-}$, $I_{+} = \psi_{1}^{\ast}(I_{-})$ and 
\[
F = \int_{0}^{1} \psi_{t}^{\ast}( d_{-}^{c} d f_{t}) dt,
\]
for $\psi_{t}$ the flow of the Hamiltonian vector field $X_{f_{t}} = Q(df_{t})$, and where $d_{-}^{c} = i (\bar{\partial}_{I_{-}} - \partial_{I_{-}})$.
\end{enumerate}
In other words, in a neighbourhood of any point, a GK structure of symplectic type is determined by a holomorphic Poisson structure, together with a time-dependent real-valued function, via the Hamiltonian flow construction of Section~\ref{hamflo}.

\end{theorem}
%
\begin{proof}
It remains only to prove the formula for $F = \lambda_{1}^{\ast} \Omega_{-}$. Differentiating the pullback $\lambda_{t}^{\ast} \Omega_{-}$ gives
\[
\frac{d}{dt}(\lambda_{t}^{\ast} \Omega_{-}) = \lambda_{t}^{\ast} \cal{L}_{Y_{t}}\Omega_{-} = \lambda_{t}^{\ast} d I^{\ast} \omega(Y_{t}) = \lambda_{t}^{\ast} d I^{\ast} t^{\ast} \alpha_{t} = \psi_{t}^{\ast} d I_{-}^{\ast} \alpha_{t} = \psi_{t}^{\ast} d_{-}^{c} \alpha_{t} = \psi_{t}^{\ast} (d_{-}^{c} d f_{t}).
\]
Therefore $\lambda_{1}^{\ast} \Omega_{-} = \int_{0}^{1} \frac{d}{dt}(\lambda_{t}^{\ast} \Omega_{-}) dt = \int_{0}^{1} \psi_{t}^{\ast}( d_{-}^{c} d f_{t}) dt$.
\end{proof}

\begin{example}
Example \ref{Kahlerex} of a K\"ahler structure showed how the generalized K\"ahler potential of Section~\ref{secpot} generalizes the usual K\"ahler potential. However, another look at this example shows that what we were doing is actually an instance of what is described in the present section. Namely, given a holomorphic Lagrangian $L$, viewed as a section $\lambda$, we used the groupoid structure to induce a local isomorphism of Morita equivalences:
\[
(T^{\ast}X, \Omega_0) \to (Z, \Omega), \qquad \alpha_{x} \mapsto \alpha_{x} + \lambda(x),
\]
so that our Darboux chart in this case is actually a groupoid chart in the sense of~\eqref{grpcht}. Now we view the brane bisection $\cal{L}$ in this chart as the graph of a $(1,0)$-form $\eta = I^{\ast} \alpha + i \alpha$, where $\alpha$ is a closed $1$-form. For our interpolating family of branes, we may choose $\cal{L}_{t} = Gr(t \eta)$. Then the induced family of closed $1$-forms on $X$ is given by the time-independent form $\alpha$. As in example \ref{Kahlerex}, we choose a real-valued function $K$ on $X$ such that $\alpha = - \frac{1}{2} dK$. Then noting that the Hamiltonian vector field $X_{K} = 0$, and appealing to theorem \ref{GKlocallyflow}, we see that the K\"ahler form is given by 
\[
\omega = -\frac{1}{2} d^{c} d K = i \partial \bar{\partial} K.
\]
Therefore the construction of Theorem \ref{GKlocallyflow} provides an alternate generalization of the K\"ahler potential which is more closely adapted to the underlying groupoid structure.
\end{example}

\bibliographystyle{hyperamsplain} \bibliography{references}

\end{document}